\documentclass[11pt,reqno,a4paper]{amsart}

\usepackage{color}
\usepackage{ulem}
\usepackage{amssymb,epsfig,graphics}

\newtheorem{theorem}{Theorem}[section]
\newtheorem{proposition}[theorem]{Proposition}
\newtheorem{hypothesis}[theorem]{Hypothesis}
\newtheorem{definition}[theorem]{Definition}
\newtheorem{corollary}[theorem]{Corollary}
\newtheorem{lemma}[theorem]{Lemma}
\newtheorem{sub-lemma}[theorem]{Sub-Lemma}

\newtheorem{remark}[theorem]{Remark}

\def\A{\mathcal{A}}

\def\M{\mathcal{M}}
\def\N{\mathcal{N}}
\def\R{\mathcal{R}}
\def\V{\mathcal{V}}

\def\EE{\mathbb{E}}

\def\PP{\mathbb{P}}
\def\RR{\mathbb{R}}

\DeclareMathOperator{\diam}{diam}
\DeclareMathOperator{\dist}{dist}
\DeclareMathOperator{\length}{length}

\def\M{\mathcal{M}}

\let\eps=\varepsilon

\def\RR{{\mathbb R}}
\def\1{{{\mathit 1} \!\!\>\!\! I} }
\renewcommand{\liminf}{\mathop{{\underline {\hbox{{\rm lim}}}}}}
\renewcommand{\limsup}{\mathop{{\overline {\hbox{{\rm lim}}}}}}

\begin{document}

\title[Spatio-temporal Poisson processes]{Spatio-temporal Poisson processes for visits to small sets}
\author{Fran\c{c}oise P\`ene \and Beno\^\i t Saussol}
\address{1)Universit\'e de Brest, Laboratoire de
Math\'ematiques de Bretagne Atlantique, CNRS UMR 6205, Brest, France\\
2)Fran\c{c}oise P\`ene is supported by the IUF.}
\email{francoise.pene@univ-brest.fr}
\email{benoit.saussol@univ-brest.fr}
%\urladdr{}
\begin{abstract}
For many measure preserving dynamical systems $(\Omega,T,m)$ the successive hitting times to a small set is well approximated by a Poisson process on the real line. In this work we define a new process obtained from recording not only the successive times $n$ of visits to a set $A$, but also the
position $T^n(x)$ in $A$ of the orbit, in the limit where $m(A)\to0$. 

We obtain a convergence of this process, suitably normalized, to a Poisson point process in time and space under some decorrelation condition. We present several new applications to hyperbolic maps and SRB measures, including the case of a neighborhood of a periodic point, and some billiards such as Sinai billiards, Bunimovich stadium and diamond billiard.
\end{abstract}
\date{\today}
\maketitle
\bibliographystyle{plain}
\tableofcontents

\section{Introduction}

The study of Poincar\'e recurrence in dynamical systems such as occurrence of rare events, distribution of return time, hitting time and Poisson law has grown to an active field of research, in deep relation with extreme values of stochastic processes;  (see~\cite{book} and references therein).

Let $(\Omega,\mathcal F,\mu,T)$ be a probability preserving dynamical system.
For every $A\in\mathcal F$, we set $\tau_A$ as the first hitting time to A, i.e.
$$\tau_A(x):=\inf\{n\ge 1\, :\, T^nx\in A\} \, .$$
In many systems the  behavior of the successive visits of a typical orbit $(T^nx)_n$ to the sets
$A_\varepsilon$, with $\mu(A_\varepsilon)\rightarrow 0+$ is often asymptotic, when suitably normalized, to a Poisson process. 
Such results were first obtained by Doeblin \cite{dob} for the Gauss map, Pitskel \cite{pit} considered the case of Markov chains. The most recent developments concern non uniformly hyperbolic dynamical systems, for example \cite{stv,cc,fft,hv,wasilewska} just to mention a few of them. 

An important issue of our work is that we take into account not only the times of successive visits to the set, 
but also the position of the successive visits in $A_\eps$ within each return. 
This study was first motivated by a question asked to us by D. Sz\'asz and I.P. T\'oth for diamond billiards, that we address in Section~\ref{sec:bd}. Beyond its own interest,
Poisson limit theorems for such spatio-temporal processes have been recently use to prove convergence to L\'evy stable processes in dynamical systems; See~\cite{marta} and subsequent works such as~\cite{MZ}. Let us indicate that, at the same time and independently
of the present work, analogous processes have been investigated
in \cite{ffm}.

We thus consider these events in time and space 
$$\{(n,T^nx)\ :\ n\ge 1,\ T^nx\in A_\varepsilon\}\subset
[0,+\infty)\times \Omega,$$
that we will normalize both in time and space, as follows.

\includegraphics[scale=0.6]{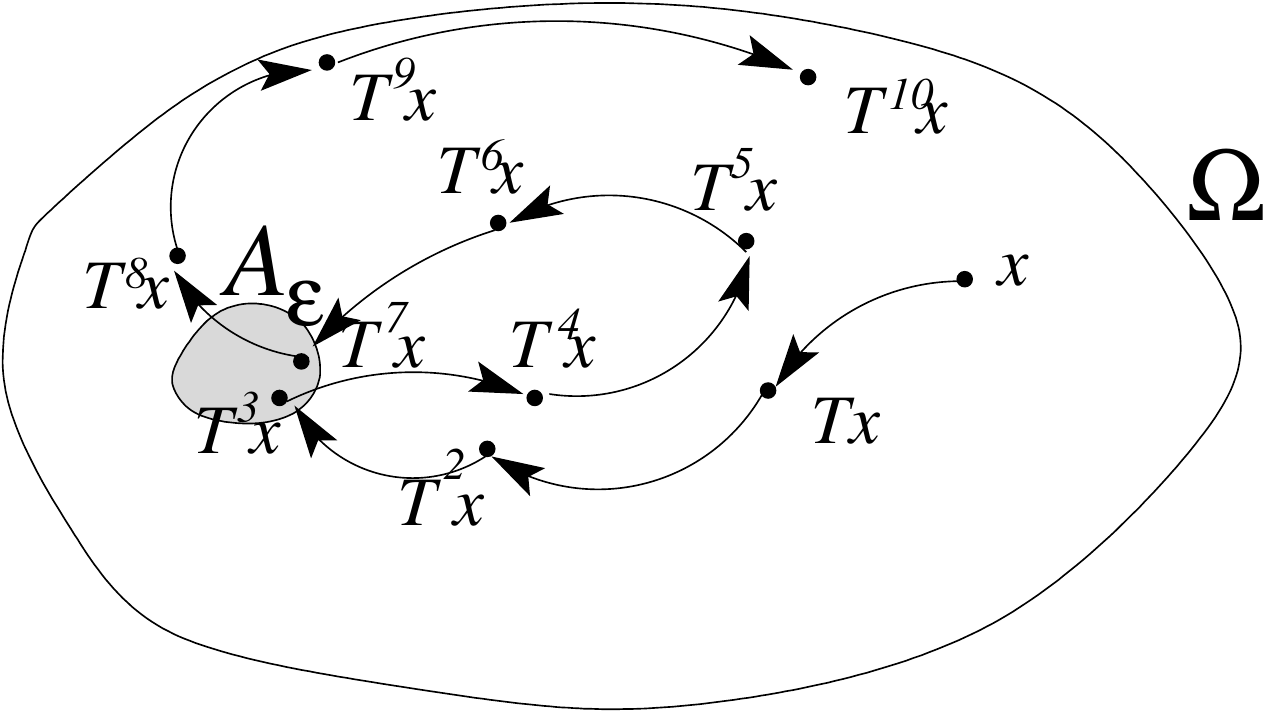}

The successive visit times have order $1/\mu(A_\varepsilon)$, which gives the normalization in time. 
For the space, we use a family of normalization functions $H_\varepsilon:A_\varepsilon\rightarrow V$. A typical choice when $\Omega$ is Euclidean would be to take for $A_\eps$ an $\eps$-ball an $H_\eps$ a zoom which sends $A_\eps$ to size one.
Another choice of extremal processes flavor would be to consider $A_\eps$ as a rare event, and $H_\eps$ would be the strength of the event.
We will then consider the family of point processes $(\mathcal N_\varepsilon)_\varepsilon$ on $[0,+\infty)\times V$ given by
\begin{equation}\label{PointProcess}
 \mathcal N_\varepsilon(x):=\sum_{n\ge 1\ :\ T^n(x)\in A_\varepsilon} \delta_{(n\mu(A_\varepsilon),H_\varepsilon(T^n(x)))} .
\end{equation}

\includegraphics[scale=0.45]{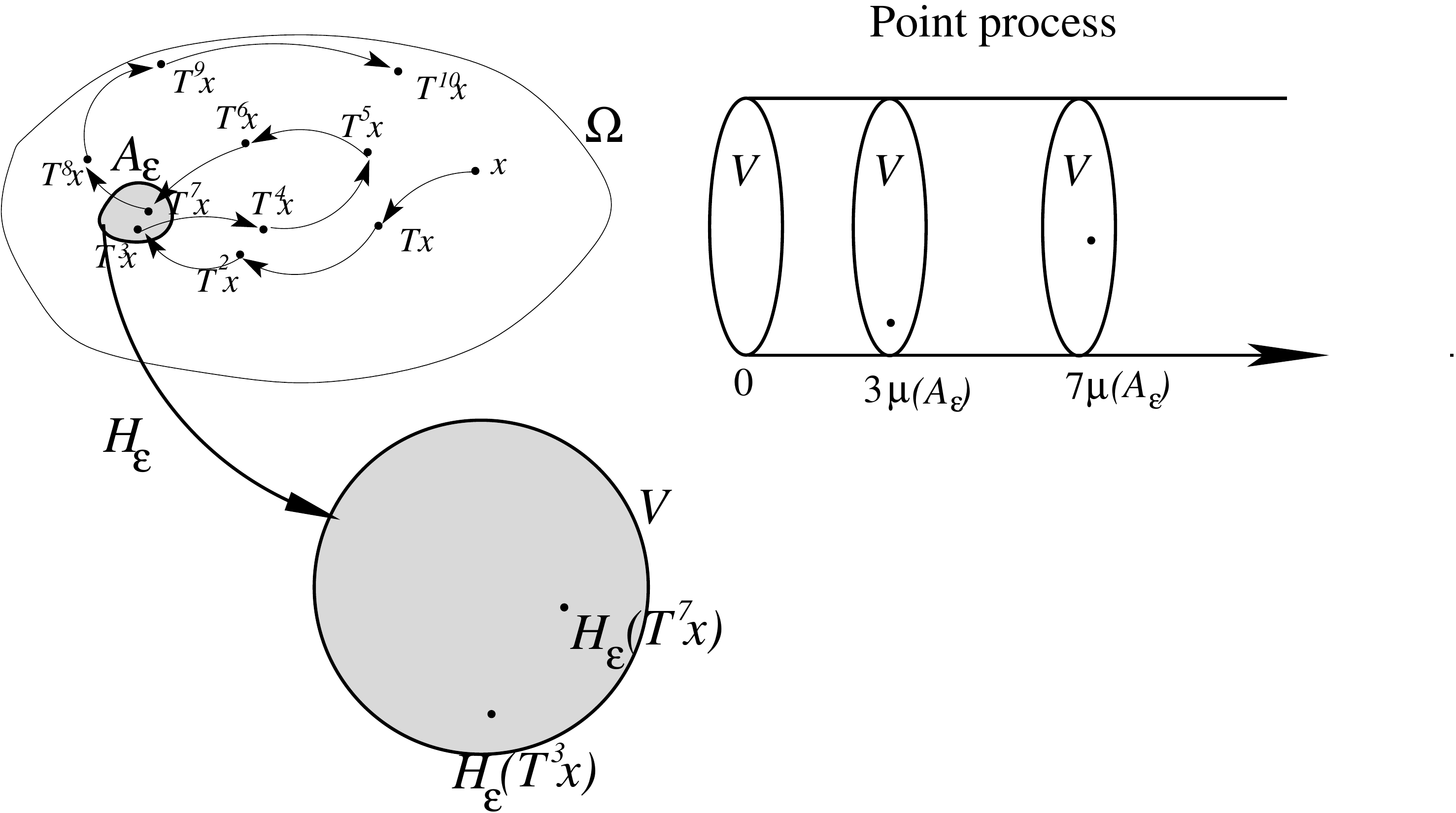}

For any measurable subset $B$ of $[0,+\infty)\times V$,
$$  \mathcal N_\varepsilon(x) (B)=\#\{n\ge 1\ :\ T^n(x)\in A_\varepsilon,\ (n\mu(A_\varepsilon),H_\varepsilon(T^n(x)))\in B\}.$$
We will simply write $\mathcal N_\varepsilon(B)$ for the
measurable function $\mathcal N_\varepsilon(B):x\mapsto \mathcal N_\varepsilon(x) (B)$.

The main result of the paper provides general conditions under which the point process
$\mathcal N_\varepsilon$ is well estimated by a Poisson point process $\mathcal P_\varepsilon$. 
This virtually contains all spatial information given by the space coordinate $H_\eps$, and the continuous mapping theorem could in principle be applied to recover many properties related to recurrences in $A_\eps$.  We then present several applications, with different maps, flows, and sets $A_\eps$.

The structure of the paper is as follows:
 
In Section~\ref{sec:gr} we present a general result which gives a convergence to a spatio-temporal point process in a discrete time dynamical system, under some probabilistic one-step decorrelation condition. Then we provide a method to transfer this result to continuous time flows. All these general results are proved in Subsection~\ref{pgr}.

In Section~\ref{common} we introduce a framework, adapted for systems modeled by a Young tower, under which one can apply the general results. 
In Subsection~\ref{sec:voisin} we check these conditions in the case of hitting to balls in a Riemmanian manifold.

In Section~\ref{billiard} we present several applications of the common framework mentioned above to two different types of billiards: the Sina\"{\i} billiards with finite horizon, the Bunimovich stadium billiards.

In Section~\ref{periodic} we study the case of balls centered around a periodic point in a uniformly hyperbolic system; as a byproduct we recover a compound Poisson distribution for the temporal process.

Section~\ref{sec:bd} consists in a fine study of visits in the successive visits in the vicinity of the corner in a diamond shaped billiard.

\section{Poisson process under a one-step decorrelation assumption}\label{sec:gr}

\subsection{Results for discrete-time dynamical systems}\label{sec:gene}
Let $(\Omega,\mathcal F,\mu,T)$ be a probability preserving dynamical system.
Let $(A_\varepsilon)_\varepsilon$ be a family of measurable subsets of $\Omega$
with $\mu(A_\varepsilon)\rightarrow 0+$ as $\varepsilon\rightarrow 0$.
Let $V$ be a locally compact metric space endowed with its Borel $\sigma$-algebra $\mathcal V$.
Let $(H_\varepsilon)_\varepsilon$ be a family of measurable functions
$H_\varepsilon:A_\varepsilon\rightarrow V$.
We set $E:=[0,+\infty)\times V$ and we endow it with its Borel
$\sigma$-algebra $\mathcal E=\mathcal B([0,+\infty))\otimes \mathcal V$.
We also consider the family of measures $(m_\eps)_\eps$ on $(V,\mathcal V)$ defined by
\begin{equation}
m_\eps:=\mu(H_\eps^{-1}(\cdot)|A_\eps) 
\end{equation}
and 
$\mathcal W$ a family stable by finite unions and intersections of relatively compact open 
subsets of $V$, that generates the $\sigma$-algebra $\mathcal V$.
Let $\lambda$ be the Lebesgue measure on $[0,\infty)$.

We will approximate the point process defined by \eqref{PointProcess} by a Poisson point process
on $ E$.
Given $\eta$ a $\sigma$-finite measure on $(E,\mathcal E)$, 
recall that a process $\mathcal N$ is a Poisson point process on $E$ of intensity $\eta$ if
\begin{enumerate}
\item $\mathcal N$ is a point process (i.e. $\mathcal N=\sum_i\delta_{x_i}$ with $x_i$ $E$-valued random variables),
\item For every pairwise disjoint Borel sets $B_1,...,B_n\subset E$, the random variables
$\mathcal N(B_1),...,\mathcal N(B_n)$ are independent Poisson
random variables with respective parameters $\eta(B_1),...,\eta(B_n)$. 
\end{enumerate}
Let $M_p(E)$ be the space of all point measures defined on $E$, endowed with the topology of vague convergence ; it is metrizable as a complete separable metric space.
A family of point processes $(\mathcal N_\eps)_\eps$ converges in distribution to $\mathcal N$ if for any bounded continuous function $f\colon  M_p(E)\to \RR$ the following convergence holds true 
\begin{equation}\label{cvdf}
\EE(f(\mathcal N_\eps))\to \EE( f(\mathcal N)),\quad \mbox{as }\varepsilon\rightarrow 0.
\end{equation}

For a collection $\mathcal A$ of measurable subsets of $\Omega$,
we define the following quantity:
\begin{equation}\label{defDelta}
\Delta(\A)
 := \sup_{A\in \A,B\in\sigma(\cup_{n=1}^{\infty}T^{-n}\A)} \left|\mu(A\cap B)-\mu(A)\mu(B)\right|.
\end{equation}
Our main general result is the following one.

\begin{theorem}\label{THM}
We assume that 
\begin{enumerate}
\item for any finite subset $\mathcal W_0$ of $\mathcal W$
we have $\Delta(H_\eps^{-1}\mathcal W_0)=o(\mu(A_\varepsilon))$,
\item there exists a measure $m$ on $(V,\mathcal V)$
such that for every $F\in\mathcal W$, $m(\partial F)=0$ and 
   $\lim_{\varepsilon\to 0}\mu(H_\eps^{-1}(F)|A_\eps)$ converges
 to $m(F)$.
\end{enumerate}
Then the family of
point processes $(\mathcal N_\varepsilon)_\varepsilon$
converges strongly\footnote{i.e. with respect to any probability measure absolutely continuous w.r.t. $\mu$} in distribution
 to a Poisson process $\mathcal P$ of intensity $\lambda\times m$. 

In particular, for every
relatively compact open $B\subset E$ such that $(\lambda\times m)(\partial B)=0$,
$(\mathcal N_\varepsilon(B))_\varepsilon$ converges in distribution
to a Poisson random variable with parameter $(\lambda\times m)(B)$. 
\end{theorem}
We emphasize that Theorem~\ref{THM} remains valid when $\eps$ is restricted to a 
subsequence $\eps_k\to0$, in the assumptions and the conclusion.

Condition (ii) is equivalent to the fact that the family of measures 
$(m_\eps)_\eps$
converges vaguely\footnote{This means that for any continuous test function $\varphi$ with compact support the integrals $m_{\eps}(\varphi)$ converge to $m(\varphi)$. In particular $m$ may not be a probability because of a loss of mass at infinity.} to $m$. This is sometimes too strong, especially when the $m_\eps$ are not absolutely continuous.
Nevertheless, without this hypothesis, our point process $\mathcal N_\varepsilon$
may remain well approximated by a Poisson process $\mathcal P_\varepsilon$ with varying intensities $\lambda\times m_\eps$, in a sense that can be made precise.
 
\begin{theorem}\label{THM1}
We assume that for any vague limit point $m$ of $(m_\eps)_\eps$ and for any sequence $\mathcal E=(\eps_k)_k$ converging to $0$ achieving the above limit,
there exists a family $\mathcal W_\mathcal E$ of relatively compact open 
subsets of $V$, stable by finite unions and intersections, that generates the $\sigma$-algebra $\mathcal V$
and that
\begin{enumerate}
\item for any finite subset $\mathcal W_0$ of $\mathcal W_\mathcal E$
we have $\Delta(H_{\eps_k}^{-1}\mathcal W_0)=o(\mu(A_{\varepsilon_k}))$,
\item for every $F\in\mathcal W_\mathcal E$, $m(\partial F)=0$.
%\footnote{This implies that $m_{\eps_k}(F)$ converges to $m(F)$.}
\end{enumerate}
Then the family of point processes $(\mathcal N_\varepsilon)_\varepsilon$
is approximated strongly in distribution by a family of Poisson processes $(\mathcal P_\eps)_\eps$ of intensities $\lambda\times m_\eps$, in the sense that for any $\nu\ll \mu$ and 
for every continuous and bounded $f\colon M_p(E)\to \RR$ 
\begin{equation}\label{approx}
\EE_\nu(f(\N_\eps)) - \EE(f( \mathcal P_\eps)) \to 0.
\end{equation}
 \end{theorem}

\begin{proof}[Proof of Theorem~\ref{THM1}]
Suppose that \eqref{approx} does not hold for some $f$.
Then there exists $\vartheta>0$ and a sequence $\eps_k\to0$ such that for all $k$, 
\begin{equation}\label{badf}
|\EE_\nu(f(\N_{\eps_k})) - \EE(f( \mathcal P_{\eps_k}))|>\vartheta.
\end{equation}
Up to taking a subsequence if necessary, we may assume that $(m_{\eps_k})_k$ converges to some $m$.
Applying Theorem~\ref{THM} with the sequence $(\eps_k)_k$ we get that $(\N_{\eps_k})_k$ converges to a Poisson point process of intensity $\lambda\times m$, and $(\mathcal P_{\eps_k})_k$ as well,  which contradicts \eqref{badf}.
\end{proof}

This proof shows that the possible non convergence of the measures $m_\eps$ is not a serious problem. {\bf In the rest of the paper, we will assume without loss of generality - to simplify the exposition of our general results - that the intensity measures involved in the results converge to a unique limit}. Clearly, if it is not the case, one can always reduces the problem to that case by passing through a subsequence.
\subsection{Application to special flows}\label{sec:specialflow}
In this section we show how to pass from the discrete time setting to the continuous time one.
Given $(\Omega,\mathcal F,\mu,T)$ and an integrable function
$\tau:\Omega\rightarrow(0,+\infty)$, the special flow over  $(\Omega,\mathcal F,\mu,T)$ with roof function $\tau:\Omega\rightarrow(0,+\infty)$ 
is the flow $(\mathcal M,\mathcal T,\nu,(Y_t)_t)$ defined as follows:
$$\mathcal M:=\{(x,t)\in \Omega\times [0,+\infty)\ :\ t<\tau(x)\} \, ,$$
$\mathcal T$ is the trace in $\mathcal M$ of the product $\sigma$-algebra $\mathcal F\otimes\mathcal B([0,+\infty))$
$$\nu:=\left(\frac{\mu\times\lambda}{\int_\Omega\tau\, d\mu}\right)_{|\mathcal T}\, ,$$
where $\lambda$ is the Lebesgue measure on $[0,+\infty)$ and 
$Y_s(x,t)=(x,t+s)$ with the identification $(x,\tau(x))\equiv (T(x),0)$,
i.e.
$$Y_s(x,t)=\left(T^{n_{s}(x,t)},t+s-\sum_{k=0}^{n_{s}(x,t)
-1}\tau\circ T^k(x)\right)\, ,$$
with $n_s(x,t):=\sup\{ n\ :\ \sum_{k=0}^{n-1}\tau\circ T^k(x)\le t+s\}$ the number of visits of the orbit $(Y_u(x,t))_{u\in(0,s)}$ to $\Omega\times \{0\}$ before time $t$.
We will write
\[
\forall x\in \Omega,\quad \forall n\in\mathbb N,\quad
S_n\tau(x):=\sum_{k=0}^{n-1}\tau\circ T^k(x)\quad\mbox{and}\quad
   \bar\tau:=\int_\Omega \tau(x)\, d\mu(x).
\]
We define also the canonical projection $\Pi:\mathcal M\rightarrow \Omega$
by $\Pi(x,t)= x$.

Let $(\mathcal A_\varepsilon)_\varepsilon$ be a sequence of subsets of $\mathcal M$. 
We are interested in a process that records the times where at least one entrance into $\mathcal A_\eps$ occurs between two consecutive returns to the base. This only depends on the projection $A_\varepsilon:=\Pi \mathcal A_\varepsilon$. 
Let $V$ be a locally compact metric space endowed with its Borel $\sigma$-algebra $\mathcal V$.
Let $(H_\eps)_\eps$ be a family of measurable functions from $A_\eps$ to $V$. 
This leads us to the definition
$$ \mathfrak N_\varepsilon 
(y):=\sum_{t>0\ :\ Y_t(y) \in A_\eps\times\{0\} }
    \delta_{(t\mu( A_\varepsilon)/\bar\tau,H_\varepsilon(\Pi Y_t(y)))}. 
    $$

\begin{theorem}\label{THMflow}
Assume that $(\mathcal N_\varepsilon)_\varepsilon$, defined on $(\Omega,\mu,T)$
by \eqref{PointProcess} with $A_\varepsilon:=\Pi \mathcal A_\varepsilon$ and $H_\varepsilon$ given above, converges in distribution, with respect to some probability measure $\tilde\mu\ll\mu$, to a Poisson point process of intensity $\lambda\times m$, where $m$ is some measure on $(V,\mathcal V)$.

Then the family of point processes $(\mathfrak N_\varepsilon)_\varepsilon$
converges strongly in distribution to a Poisson process $\mathcal P$ of intensity $\lambda\times m$.
\end{theorem}

\subsection{Proofs of the general theorems}\label{pgr}
In this section we prove Theorems \ref{THM} and \ref{THMflow}.

To show that the family of point processes $(\N_\eps)_\eps$ on $\RR^+\times V=E$ converges (with respect to some probability measure $\mathbb P$) to a Poisson point process with a $\sigma$-finite intensity $\eta$, we can apply Kallenberg's criterion (Proposition 3.22 in \cite{resnick}).
It suffices to prove that for some system $\mathcal R$ stable by finite intersection and union of relatively compact open 
subsets, that generates the $\sigma$-algebra $\mathcal V\times \mathcal B(\RR^+)$, the following holds:

For any $R\in\R$,
\begin{itemize}
\item[(O)] $\eta(\partial R)=0$,
\item[(A)] $\EE(\N_\eps(R))\to \eta(R)$,
\item[(B)] $\PP(\N_\eps(R)=0)\to e^{-\eta(R)}$.
\end{itemize}
The last condition will be obtained from the simple next geometric approximation.
\begin{proposition}\label{pro:delta}
Let $\A$ be a collection of measurable subsets of $\Omega$.
Then for any $r\ge 1$, any positive integers $p_1,\ldots, p_r,q_1,\ldots,q_r$ such that
$p_i+q_i<p_{i+1}$ for any $i=1,...r-1$, and for any sets $A_1,\ldots, A_r\in\A$
we have
\[
\left|\mu(\forall i,\tau_{A_i}\circ T^{p_i}>q_i)-\prod_{i=1}^r(1-\mu(A_i))^{q_i}\right| 
\le  \sum_{i=1}^r q_i  \Delta(\A)\, ,
\]
with $\Delta$ defined in \eqref{defDelta}.
\end{proposition}

\begin{proof}
We will write $a = b\pm c$ to say that $|a-b|\le c$.
Note that for any integer $q_1\ge1$ we have 
\[
\{ \tau_{A_1}>q_1 \} = T^{-1} \{ \tau_{A_1}>q_1-1 \} - T^{-1}(A_1\cap \{ \tau_{A_1}>q_1-1 \}).
\]
Using the $T$-invariance of $\mu$, this gives
\[
\begin{split}
&\mu(\forall i,\ \tau_{A_i}\circ T^{p_i}>q_i)\\
&=\mu(\forall i,\ \tau_{A_i}\circ T^{p_i-p_1}>q_i)\\
&=\mu(\tau_{A_1}>q_1-1 ; \forall i\ge2,\tau_{A_i}\circ T^{p_i-p_1-1}>q_i)\\
&\quad -\mu(A_1\cap \{\tau_{A_1}>q_1-1 ; \forall i\ge2,\tau_{A_i}\circ T^{p_i-p_1-1}>q_i\}) \\
&=(1-\mu(A_1)) \mu(\tau_{A_1}>q_1-1 ; \forall i\ge2,\tau_{A_i}\circ T^{p_i-p_1+1}>q_i)\pm \Delta(\A).
\end{split}
\]
By an immediate induction on $q_1$ we obtain
\[
\mu(\forall i\ge 1,\tau_{A_i}\circ T^{p_i}>q_i) = (1-\mu(A_1))^{q_1} \mu(\forall i\ge2,\tau_{A_i}\circ T^{p_i-q_1}>q_i)\pm q_1\Delta(\A).
\]
The conclusion follows by an induction on the number $r$ of sets.
\end{proof}

\begin{proof}[Proof of Theorem \ref{THM}]
The strong convergence in distribution is by Theorem 1 in \cite{Zweimuller:2007} a direct consequence of the classical convergence in distribution with respect to $\mu$.
The later will be proven as announced using Kallenberg's criterion
with $\eta:=\lambda\times m$.

Let $\mathcal{R}$ be the collection of finite unions of open rectangles $I\times F$, where $I$ is an open bounded interval in $[0,\infty)$ and $F\in\mathcal W$.

Let $R\in\mathcal{R}$. We rearrange the subdivision given by the endpoints of the intervals defining $R$ to write $R=R' \cup  \bigcup_{i=1}^r (t_i,s_i)\times F_i $, where $t_i< s_i$,  $F_i\in\mathcal{W}$ for $i=1,...,r$,  $s_i\le t_{i+1}$ for every $i=1,...,r-1$ and $R'$ is contained in a finite number of vertical strips $\{t\}\times V$.
We assume without loss of generality that $R'=\emptyset$, as it will be clear from the sequel that the same arguments would give $\mathbb E_\mu[\mathcal N_\varepsilon(R')]\to0$ and $\mu(\N_\eps(R')=0)\to1$.

Note that for all $i=1,\ldots ,r$, $m_\eps(F_i)\to m(F_i)$ as $\eps\to 0$ by  Hypothesis (ii) and the Portemanteau theorem.
 
Condition (A) follows from the definition and the linearity of the expectation.
Indeed,
$$\mathbb E_\mu[\mathcal N_\varepsilon(R)]=\sum_{i=1}^r \sum_{n=\lfloor\frac {t_i}{\mu(A_\varepsilon)}\rfloor+1}^{\lceil\frac {s_i}{\mu(A_\varepsilon)}\rceil-1}\mu(H_\varepsilon^{-1}(F_i))\sim\sum_i(s_i-t_i)m_\eps(F_i)\sim (\lambda\times m)(R).$$

For Condition (B), set $p_i=\lfloor t_i/\mu(A_\eps)\rfloor$ and $q_i=\lceil s_i/\mu(A_\eps)\rceil-1-p_i$ and $A_{\eps,i}=H_\eps^{-1}F_i$.
Observe that $\N_\eps(R)=0$ is equivalent to 
$\forall i, \tau_{A_{\eps,i}}\circ T^{p_i}>q_i$. By Proposition~\ref{pro:delta} and due to assumption (i),
we have
\begin{eqnarray*}
\mu(\N_\eps(R)=0) &=& \prod_i (1-\mu(A_{\eps,i}))^{q_i} \pm \sum_{i=1}^r q_i \Delta(\{A_{\varepsilon,i}\})\\
&=& \prod_i (1-\mu(A_\varepsilon)m_\eps(F_i))^{q_i} + o(1)\\
&=&\exp\left(\sum_i  q_i \mu(A_\varepsilon) m_\eps(F_i)\right) +o(1)\\
&=&\exp\left((\lambda\times m)(R) \right)+o(1).
\end{eqnarray*}
\end{proof}

\begin{proof}[Proof of Theorem \ref{THMflow}]
Due to Theorem 1 of \cite{Zweimuller:2007}, the assumption holds also with $\tilde\mu$ s.t.  $\dfrac{d\tilde\mu}{d\mu}=\dfrac{\tau}{\tilde\tau}$, and it suffices to prove the convergence in distribution of 
$(\mathcal N_\varepsilon)_\varepsilon$ with respect to $\nu$.

By hypotheses, 
$(\mathcal N_\varepsilon)_\varepsilon$
converges in distribution wrt $\Pi_*\nu$ (which coincide with
the probability measure 
on $\Omega$ with density $\tau/\bar\tau$ wrt $\mu$) to a
Poisson point process $\mathcal P$ of intensity $\lambda\times m$,
i.e. $(\mathcal N_\varepsilon\circ\Pi)_\varepsilon$ converges in distribution, with respect to $\nu$, to $\mathcal P$. 

Note that for $y=(x,s)\in \M$
\[
\mathfrak N_\varepsilon(y)= (\psi_{\varepsilon,y})_*
   (\mathcal N_\varepsilon(\Pi(y)))\ 
\]
with
\[
\psi_{\varepsilon,y}(t,z)=
((S_{\lfloor t/\mu(A_\varepsilon)\rfloor}\tau(\Pi y)-s)\mu(A_\varepsilon)/\bar\tau,z)
\]
We conclude using the facts
 that $\mu(A_\varepsilon)\to 0$ and that
$S_n\tau/(n\bar\tau)\to 1$ $\nu$-almost surely.
\end{proof}
\section{A common framework suitable for systems modeled by Gibbs-Markov-Young towers}\label{common}

\subsection{General result}
The authors have studied in \cite{ps15} the case of a temporal Poisson 
point process, corresponding to $V=\{0\}$ and $H_\varepsilon\equiv 0$ (Theorem 3.5 therein appears as Theorem \ref{THM} of the present paper in this specific context).
In the above mentioned article, the general context was the case of dynamical systems modeled by a Gibbs Markov Young tower as studied in~\cite{aa}.
\begin{hypothesis}\label{HHH}
Let $\alpha,\beta>0$.
Let $\Omega$ be a metric space endowed with a Borel
probability measure $\mu$ and a $\mu$-preserving transformation
$T$. Assume that there exists a sequence of finite partitions $(\mathcal Q_k)_{k\ge 0}$ of an extension $(\tilde\Omega,\tilde\mu,\tilde T)$
of $(\Omega,\mu,T)$ by $\tilde\Pi:\tilde\Omega\rightarrow\Omega$ such that one of the two following assumptions holds:
\begin{itemize}
\item[(I)] either $\sup_{Q\in\mathcal Q_k}\diam (\tilde\Pi(Q))\le C k^{-\alpha}$,
\item[(II)] or $\sup_{Q\in\mathcal Q_{2k}}\diam(\tilde\Pi \tilde T^kQ)\le C k^{-\alpha}$.
\end{itemize}
Assume moreover that, there exists $C'>0$ such that, for every $k,n$
with $n\ge 2k$, for any $A\in \sigma(\mathcal Q_k)$
and for any $B\in\sigma\left(\bigcup_{m\ge 0}\mathcal Q_m\right)$,
$$\left|Cov_{\tilde\mu}\left(\mathbf 1_A,\mathbf 1_{B}\circ \tilde T^n\right)\right|\le C' n^{-\beta}\tilde\mu(A).$$
\end{hypothesis}
Case (I) is appropriate for non-invertible systems, whereas Case (II)
is appropriate for invertible systems.
Wide classes of examples of dynamical systems satisfying these
assumptions can be found in \cite{aa} (see also
 \cite{young,young99}).

\begin{proposition}\label{appli1}
Let $(\Omega,\mathcal F,\mu,T)$ satisfying Hypothesis \ref{HHH}.
With the notations of the beginning of Section \ref{sec:gene},
assume that there exists $p_\eps=o(\mu(A_\eps)^{-1})$ such that
\begin{enumerate}
\item\label{a}
 $  \mu(\tau_{A_\eps}\le p_\varepsilon | A_\eps)=o(1)$,
\item \label{ii}
$\mu((\partial A_\varepsilon)^{\lfloor C4^\alpha p_\varepsilon^{-\alpha}\rfloor})=o(\mu(A_\varepsilon))$ 
\item \label{iii} $\mu(H_\varepsilon^{-1}(\cdot)|A_\varepsilon)$ converges vaguely to
some measure $m$,
\item\label{F} 
for all $F\in\mathcal W$, $m(\partial F)=0$ and $\mu((\partial (H_\varepsilon^{-1} F))^{\lfloor C4^\alpha p_\varepsilon^{-\alpha}\rfloor}|A_\varepsilon)=o(1)$.
\end{enumerate}
Then the assumptions of Theorem~\ref{THM} are satisfied.
\end{proposition}
We postpone the proof to the appendix.
When assumption  \eqref{iii} is missing, one has to change \eqref{F} accordingly, going through a subsequence $\mathcal E=(\eps_k)$, a family $\mathcal W_{\mathcal E} $ and a limit point $m$ to get that the assumptions of Theorem~\ref{THM1} are satisfied. 

\begin{remark}\label{first}
\begin{itemize}
\item
Notice that the Assumption \eqref{a} is always satisfied when the normalized first return time $\mu(A_\eps)\tau_{A_\eps}$ to $A_\eps$ is asymptotically exponentially distributed with parameter one, in particular when the temporal process (corresponding to $V=\{0\}$) converges in distribution to a Poisson process of intensity $\lambda$. 
\item
It may happen that one only has a nonuniform control of the diameters instead of Hypothesis \ref{HHH} (I) or (II).
Indeed it is possible to avoid these hypotheses. It suffices to replace (ii) with  
$$
\mu(\partial A_\eps^{[k_\eps]}) = o(\mu(A_\eps)),
$$
where 
in the non invertible case (I) $\partial A_\eps^{[k_\eps]}=\cup\tilde\Pi(Q)$, the union begin on those $Q\in \mathcal{Q}_{k_\eps}$ such that $\diam \tilde{\Pi}(Q)>Ck_\eps^{-\alpha}$ and $\partial A_\eps\cap\tilde{\Pi}(Q)\neq\emptyset$, $k_\eps =\lfloor p_\eps/2\rfloor$;
in the invertible case (II)  $\partial A_\eps^{[k_\eps]}=\cup\tilde\Pi(T^{k_\eps}Q)$, the union begin on those $Q\in \mathcal{Q}_{2k_\eps}$ such that $\diam \tilde{\Pi}(T^{k_\eps}Q)>Ck_\eps^{-\alpha}$ and $\partial A_\eps\cap\tilde{\Pi}(T^{k_\eps} Q)\neq\emptyset$, $k_\eps=\lfloor p_\eps/4\rfloor $; 
The modification of the proof of Proposition~\ref{appli1} is immediate.
\end{itemize}
\end{remark}

The following result is helpful to get assumption~\eqref{F} in many cases. 
\begin{proposition}\label{partition}
Assume that $V$ is an open subset of $\RR^d$, for some $d>0$,
that $H_\eps$ are $\mathfrak h$-H\"older continuous maps with respective H\"older constant $C_{\mathfrak h}(H_\eps)$, that there exists $\eta_\eps\rightarrow 0+$ such that $C_{\mathfrak h}(H_\eps)\eta_\eps^{\mathfrak h}\to0$ and that $m_\eps\to m$ (where $m$ is a finite measure on $V$). 
Then there exists a family $\mathcal W$ of relatively compact open subsets of $B$, stable by finite unions and intersections, generating the Borel $\sigma$-algebra of $V$, such that
\begin{enumerate}
\item
$m(\partial F)=0$ for any $F\in \mathcal W$,
\item for any $F\in \mathcal W$,
$$ 
\mu((\partial H_\eps^{-1}F)^{[\eta_\eps]} |A_\eps) = o(1).
$$
\end{enumerate}

\end{proposition}
\begin{proof}
(i) 
Let $\pi_j:\mathbb R^d\rightarrow\mathbb R$ be the $j$-th canonical projection (i.e. $\pi_j((x_i)_i)=x_j$).
The set $\mathcal G^{j}:=\{ a\in\RR\colon (\pi_j)_*m({a})=0\}$ is dense in $\RR$, 
since its complement is at most countable.

We then define $\mathcal W$ as the collection
% of finite unions 
of open rectangles $\prod_{j=1}^d(a_j;b_j)\subset B$, with $a_j,b_j\in\mathcal G^{j}$.
By construction $m(\partial F)=0$ for any $F\in \mathcal W$.
The density of the $\mathcal G^j$'s implies that $\mathcal W$ generates the Borel $\sigma$-algebra.

(ii)
%We suppose without loss of generality that $(L_\eps\eta_\eps)_\eps$ is monotonous.
For any $F\in\mathcal W$ we have the inclusion
\[
(\partial H_\eps^{-1}F)^{[\eta_\eps]} \subset H_\eps^{-1} \partial F^{[C_{\mathfrak h}(H_\eps)\eta_\eps^{\mathfrak h}]}.
\]
Hence, for every $\eps_0>0$,
\[
\mu((\partial H_\eps^{-1}F)^{[\eta_\eps]} |A_\eps) \le m_\eps(\partial F^{[C_{\mathfrak h}(H_\eps)\eta_\eps^{\mathfrak h}]})\le m_\eps(\partial F^{[M_{\eps_0}]}),
\]
for any $0<\eps<\eps_0$, with $M_{\eps_0}:=\sup_{\eps\in(0,\eps_0)}C_{\mathfrak h}(H_\eps)\eta_\eps^{\mathfrak h}$. Therefore
\[
\limsup_{\eps\to0} \mu((\partial H_\eps^{-1}F)^{[\eta_\eps]} |A_\eps) \le m(\partial F^{[M_{\eps_0}]}),
\]
which goes to $0$ as $\eps_0\to0$ since $m(\partial F)=0$. 
\end{proof}

\subsection{Successive visits in a small neighbourhood of a generic point}\label{sec:voisin}
The purpose of the next result is to give examples for which 
Theorem \ref{THM} applies for returns in small balls, in the same context as in \cite{ps15}. 

\begin{theorem}\label{ThmgeneAppli1}
Assume that 
$\Omega$ is a $d$-Riemmannian manifold and that $\dim_H\mu=\liminf_{\varepsilon\rightarrow 0}
\frac{\log \mu(B(x_0,\varepsilon))}{\log\varepsilon}$ for $\mu$-almost every $x_0\in\Omega$.
Assume Hypothesis \ref{HHH} with $\alpha>\dim_H\mu$.
 
Then for $\mu$-almost every $x_0\in\Omega$ such that
\begin{equation}\label{neglec}
\exists \delta\in (1,\alpha\dim_H\mu),\quad \mu(B(x_0,\varepsilon)\setminus B(x_0,\varepsilon-\varepsilon^\delta)=o(\mu(B(x_0,\varepsilon))),
\end{equation}
the family of point processes 
$$
\mathcal N_\varepsilon(x) =\sum_{n\colon T^n(x)\in B(x_0,\eps)} \delta_{\left(n\mu(B(x_0,\eps),\eps^{-1}\exp_{x_0}^{-1}(T^nx)\right)}
$$
is strongly approximated in distribution
 by a Poisson process $\mathcal P_\eps$ of intensity $\lambda\times m_\eps^{x_0}$ where 
 $ m_\eps^{x_0}=\mu(H_\eps^{-1}(\cdot)|B(x_0,\varepsilon))$. 
\end{theorem}
\begin{proof}[Proof of Theorem~\ref{ThmgeneAppli1}]
We apply Proposition~\ref{appli1}, assuming without loss of generality that the measures $m_\eps^{x_0}$ are converging.

Fix $\sigma\in(\delta/\alpha,\dim_H\mu)$ and set $p_\eps=\eps^{-\sigma}$.
The assumptions imply that the temporal return times process converges in distribution to a Poisson process of parameter one (See \cite{ps15}). By Remark~\ref{first} this shows that Assumption \eqref{a} of Proposition~\ref{appli1} is true.  
Assumption~\eqref{ii} follows from \eqref{neglec}.
Finally, the family $\mathcal W$ comes from Proposition~\ref{partition} above,
which proves also the last assumption \eqref{F} and thus the theorem.
\end{proof}

\section{Applications to billiard maps and flows}\label{billiard}
\subsection{Bunimovich billiard}\label{sec:Bunimovich}
The Bunimovich billiard is an example of weakly hyperbolic system
(with polynomial decay of the covariance of H\"older
functions).

Let $\ell>0$. We consider the planar domain $Q$ union of the rectangle
$[-\ell/2;\ell/2]\times[-1,1]$ and of the two planar discs of radius 1
centered at $(\pm\ell/2,0)$.
We consider a point particle moving with unit speed in $Q$, going
straight on between two reflections off $\partial Q$ and reflecting with
respect to the classical Descartes law of reflection (incident angle=reflected angle).
The billiard system $(\Omega,\mu,T)$ describes the evolution
at reflected times
of a point particle moving in the domain $Q$ as described above.
We define the set $\Omega$ of reflected vectors as follows
$$ \Omega:=\{\partial Q \times S^1\ :\  <\vec n_q,\vec v>\ge 0\}\, ,$$
where $\vec n_q$ is the unit normal
vector normal to $\partial Q$, directed inside $Q$, at $q$.
Such a reflected vector $(q,\vec v)$ can be represented
by $x=(r,\varphi)\in\mathbb R/(2(\pi+\ell)\mathbb Z)$,
$r$ corresponding to the counterclockwise curvilinear abscissa of $q$ on $\partial Q$ (starting from a fixed point on $\partial Q$) and 
$\varphi$ measuring the angle between $\vec n_q$ and $\vec v$.
The transformation $T$ maps a reflected vector to the reflected vector
corresponding to the next reflection time.
This transformation preserves the measure $\mu$ given (in coordinates) by
${d\mu}(r,\varphi)=h(r,\varphi) \, dr\, d\varphi$ where $h(r,\varphi)=\frac{\cos\varphi}{2\, |\partial Q|}$.
We endow $\Omega$ with the supremum metric $d((r,\varphi),(r',\varphi'))=\max(|r-r'|,|\varphi-\varphi'|)$. 
\begin{theorem}\label{thmBunimovich}
For $\mu$-almost every $x_0=(r_0,\varphi_0)\in \Omega$,
the family of point processes
$$\sum_{n\ge 1\ :\ d(T^n(x),x_0)<\eps} \delta_{\left(n\varepsilon^2h(x_0),\frac{T^n(x)-x_0}\eps\right)} $$
converges in distribution (with $x$ distributed with respect to any probability measure absolutely
continuous with respect to the Lebesgue measure on $\Omega$) to 
a Poisson Point Process with intensity $\lambda\times\lambda_2$, where here
$\lambda_2$ is the 2-dimensional normalized Lebesgue measure on $(-1,1)^2$.
\end{theorem}
\ \bigskip

\begin{center}
\includegraphics[scale=0.5]{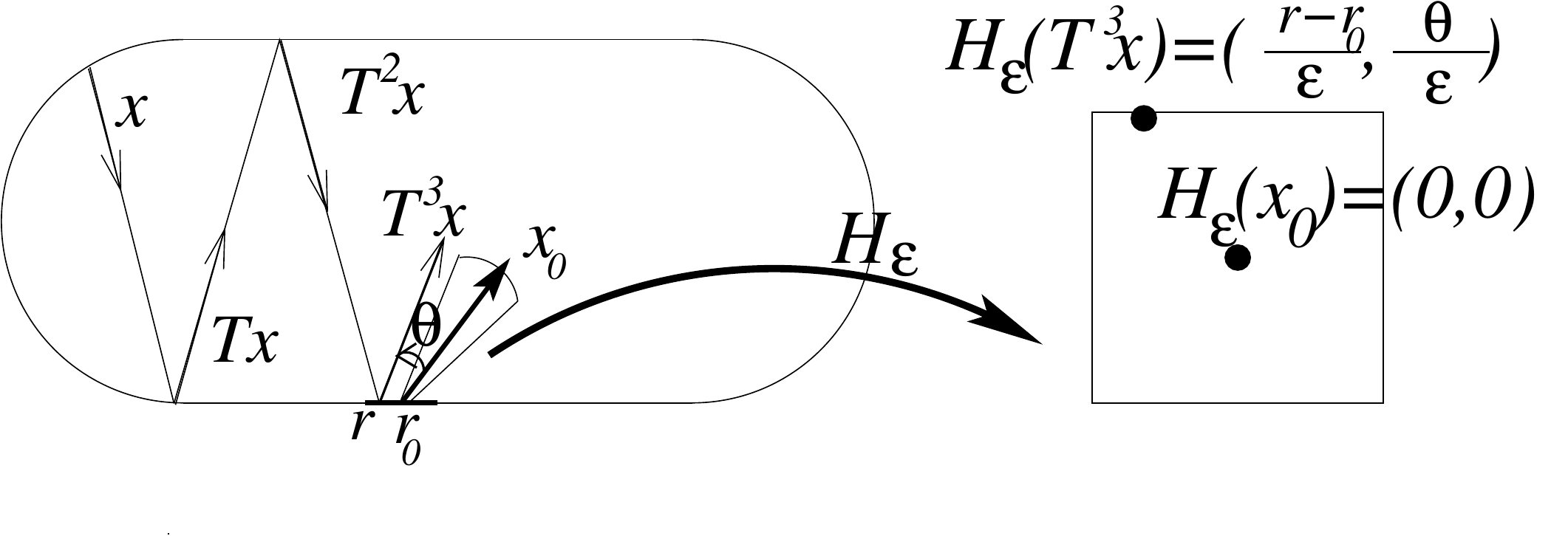}
\end{center}

\begin{proof}
The convergence comes directly from Theorem \ref{ThmgeneAppli1}, 
case (II) with $\alpha=1$, $\zeta=1$ and $\dim_H\mu=2$,
(due to \cite{ps15}, namely in Section 9 therein)
and from the fact that $\mu(H_\varepsilon^{-1}(\cdot) |A_\varepsilon)$
converges in distribution to the normalized Lebesgue measure 
on $B_{|\cdot|_\infty}(0,1)$. 
To identify the intensity we observe that $\mu(B(x_0,\varepsilon))\sim \cos\varphi_0\varepsilon^2/(2|\partial Q|)$ as $\varepsilon\rightarrow 0+$.
\end{proof}
We now consider the billiard flow $(\mathcal M,\nu,(Y_t)_t)$ with $$\mathcal M=\{(q,\vec v)\in Q \times S^1\ :\ q\in \partial Q\Rightarrow <\vec n_q,\vec v>\ge 0\},$$
where $Y_t(q,\vec v)=(q',\vec v')$ if a particle that was at time 0 at
position $q$ with speed $\vec v$ will be at time $t$ at position $q'$
with speed $\vec v'$ and where $\nu$ is the normalized Lebesgue measure on $\mathcal M$. 

An important well known fact is that the billiard flow system $(\mathcal M,\nu,(Y_t)_t)$ can be represented as the special flow over
the billiard map system $(\Omega,\mu,T)$ with roof function
$\tau:\Omega\rightarrow(0,+\infty)$ given by
$\tau(x):=\inf\{t>0\ :\ Y_t(x)\in\partial Q\times S^1\}$.
This enables us to get the following result.
\begin{corollary}\label{coroBunimovich}
For $\mu$-almost every $x_0=(r_0,\varphi_0)\in\Omega$, 
the family of point processes
$$ \sum_{t>0\ :\ Y_t(y)\in B_\Omega(x_0,\varepsilon)} \delta_{\left(\frac{t\varepsilon^2\cos\varphi_0}{2\pi Area(Q)},\frac{Y_t(y)-x_0}{\varepsilon}\right)} $$
converges in distribution (with respect to any probability measure absolutely
continuous with respect to the Lebesgue measure on $\mathcal M$) to 
a Poisson Point Process with intensity $\lambda\times\lambda_2$, where here
$\lambda_2$ is the 2-dimensional normalized Lebesgue measure on $[-1,1]^2$ and where $B_\Omega(x_0,\eps)$ means the ball in $\Omega$.
\end{corollary}
\begin{proof}
Set  $\bar\tau:=\int_\Omega\tau\, d\mu$.
Due to Theorems \ref{thmBunimovich} and \ref{THMflow}, the point process
$$\sum_{t>0\ :\ Y_t(y)\in B_\Omega(x_0,\varepsilon)} \delta_{\left(\frac{t\varepsilon^2\cos\varphi_0}{2\bar\tau|\partial Q|},\frac{Y_t(y)-x_0)}\eps  \right)} $$
converges in distribution to a Poisson Point Process with intensity
$\lambda\times\lambda_2$. Moreover $2\pi Area(Q)=\int_\Omega\tau\cos\varphi\, dr\, d\varphi=2|\partial Q|\bar\tau .$
\end{proof}
\subsection{Sinai billiards}
Let $I\in\mathbb N^*$ and $O_1,...,O_I$ be open convex subsets
of $\mathbb T^2$ with $C^3$-smooth boundary of positive curvature, and pairwise
disjoint closures. We then set $Q=\mathbb T^2\setminus\bigcup_{i=1}^IO_i$.
As for the Bunimovich billiard, we consider a point particle
moving in $Q$, with unit speed and elastic reflections off $\partial Q$.
This model is called the Sinai billiard.
We assume moreover that the horizon is finite, i.e. that the time
between two reflections is uniformly bounded.

For this choice of $Q$, we consider now the billiard map $(\Omega,\mu,T)$ and the billiard flow $(\mathcal M,\nu,(Y_t)_t)$ defined
as for the Bunimovich billiard in Subsection \ref{sec:Bunimovich}.

Under this finite horizon assumption, the Sinai billiard has
much stronger hyperbolic properties than the Bunimovich
billiard (with namely an exponential decay of the covariance of H\"older functions), but nevertheless, compared to Anosov
map, its study is complicated by the presence of discontinuities.

\begin{theorem}\label{Sinai1}
For $\mu$-almost every $x_0=(q_0,\vec v_0)\in \Omega$ (represented by $(r_0,\varphi_0)$),
\begin{itemize}
\item {\bf [Return times in a neighborhood of $x_0$]}
The conclusions of Theorem \ref{thmBunimovich} and of its Corollary
\ref{coroBunimovich} hold also true.
\item {\bf [Return times in a neighborhood of the position $q_0$ of $x_0$]}
The family of point processes
$$ \sum_{n\ge1\ :\ T^n(x)\in B_{\partial Q}(q_0,\varepsilon)\times S^1} \delta_{\left(\frac{2\varepsilon n}{|\partial Q|},\frac{r-r_0}{\eps},\varphi\right)}$$
(where $x$ is represented by $(r,\varphi)$)
converges in distribution (with respect to any probability measure absolutely
continuous with respect to the Lebesgue measure on $\Omega$) to 
a Poisson Point Process with intensity $\lambda\times m_0$, where here
$m_0$ is the probability measure on $[-1,1]\times[-\pi/2;\pi/2]$ 
with density $(r,\varphi)\mapsto \frac{\cos(\varphi)}{4}$.
\item The family of point processes
$$\sum_{t>0\ :\ Y_t(y)\in  B_Q(q_0,\varepsilon)\times S^1} 
\delta_{\left(\frac{2\varepsilon t}{\pi Area(Q)},\frac{r-r_0}\eps,\varphi\right)}$$
converges in distribution (with respect to any probability measure absolutely
continuous with respect to the Lebesgue measure on $\mathcal M$) to 
a Poisson Point Process with intensity $\lambda\times m_0$, with
$m_0$ as above.
\end{itemize}
\end{theorem}
Due to \cite{young}, the Sinai billiard $(\Omega,\mu, T)$
satisfies Hypothesis \ref{HHH}-(II) with any $\alpha>0$ and any $\beta>0$.
\begin{proof}
The first item follows from Theorem~\ref{ThmgeneAppli1} as in Theorem~\ref{thmBunimovich} and Corollary~\ref{coroBunimovich}.

The third item follows from the second one as Corollary \ref{coroBunimovich} comes from Theorem \ref{thmBunimovich}.

Let us prove the second item. 
Let $x_0=(q_0,\varphi_0)\in\Omega$.
We set $A_\varepsilon(x_0):=[q_0-\varepsilon,q_0+\varepsilon]
\times [\pi/2,\pi/2]$ and $H_\varepsilon:(q,\varphi)\mapsto(\varepsilon^{-1}(q-q_0),\varphi)$. Note first that
$\mu(A_\varepsilon)=2\varepsilon/|\partial Q|$ and that $\mu(H_\varepsilon^{-1}(\cdot) |A_\varepsilon)$
converges vaguely to $m_0$.

We follow verbatim the proof of Theorem~\ref{ThmgeneAppli1} which  invokes Proposition \ref{appli1}, except for its assumption~\eqref{a}:

The convergence of the temporal process as in the proof of Theorem~\ref{ThmgeneAppli1}, associated to the position of the billiard particle, is not present in the literature in these terms. However, a slight adaptation of \cite[Lemma 6.4-(iii)]{ps10} gives that for any $\sigma<1$
and $\mu$-a.e. $x_0\in\partial Q\times]-\frac\pi 2;\frac\pi 2[$ we have 
$$  \mu(\tau_{A_\varepsilon}\le \varepsilon^{-\sigma}|A_\varepsilon)=o(1).$$
That is (i) of Proposition \ref{appli1} holds with $p_\eps=\eps^{-\sigma}$.
\end{proof}

Let us write $\Pi_Q:\mathcal M\rightarrow Q$ and $\Pi_V:\mathcal M\rightarrow S^1$ for the two canonical projections, which correspond
respectively to the position and to the speed.
Using results established in \cite{ps10}, we also state a result of convergence to a spatio-temporal Poisson point process for entrance
times in balls for the flows.
We endow $\mathcal M$ with the metric $d$ given by 
$d((q,\vec v),(q',\vec v'))=\max(d_0(q,q'),|\angle{(\vec v,\vec v')}|)$,
where $d_0$ is the euclidean metric in $Q$ and where $\angle(\cdot,\cdot)$ is the angular measure of the angle.
\begin{theorem}
For $\nu$-a.e. $y_0=(q_0,\vec v_0)\in\mathcal M$,
\begin{itemize}
\item the family of point processes
$$\sum_{t\ :\ (Y_s(y))_s\mbox{ enters }  B(y_0,\varepsilon)
\mbox{ at time t}} \delta_{\left(\frac{2\varepsilon^2t}{\pi Area(Q)},\frac{\Pi_Q(Y_t(y))-q_0}\eps,\frac{\angle{(\vec v_0,\Pi_V(Y_t(y)))}}\eps\right)}$$
converges in distribution (when $y$ is distributed with respect to any probability measure absolutely
continuous with respect to the Lebesgue measure on $\mathcal M$) to 
a Poisson Point Process with intensity $\lambda\times \tilde m_1$, where $\tilde m_1$ is the probability measure of density $(p,\vec u)\mapsto\frac14\langle\tilde n_p,\vec v_0\rangle^+$ on $S^1\times[-1,1]$ (with $\langle\cdot,\cdot\rangle^+$ the positive part of the scalar product in $\mathbb R^2$ and $\tilde n_p$ the inward normal vector to $S^1$ at $p$).\footnote{In the limit, the authorized normalized positions are the positions located on a semi-circle (corresponding to positions at which the vector $\vec v_0$ enters the ball) and the normalized variation of speed is uniform in $[-1,1]$}

\includegraphics[scale=0.4]{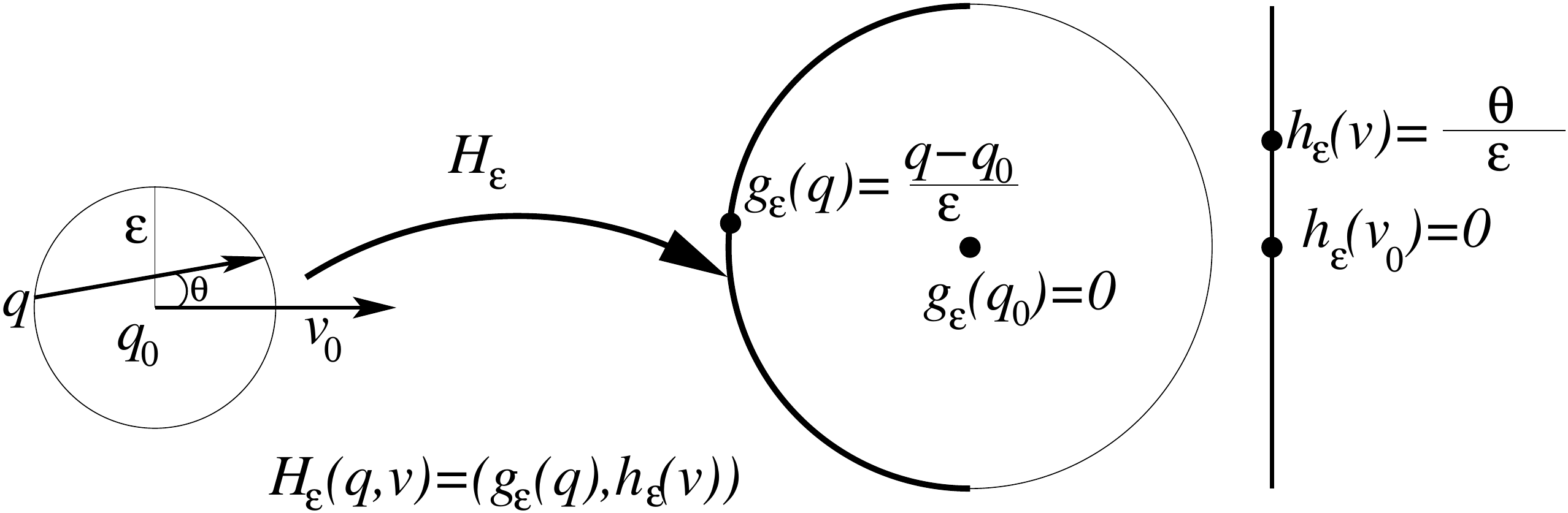}

\item the family of point processes
$$\sum_{t\ :\ (Y_s(y))_s\mbox{ enters } B(q,\varepsilon)\times S^1\mbox{ at time t}}\delta_{\left(\frac{2\pi\varepsilon t}{Area(Q)},\frac{\Pi_Q(Y_t(y))-q_0}\eps,\Pi_V(Y_t(y))\right)}$$
converges in distribution (when $y$ is distributed with respect to any probability measure absolutely
continuous with respect to the Lebesgue measure on $\mathcal M$) to 
a Poisson Point Process with intensity $\lambda\times \tilde m_0$
where $\tilde m_0$ is the probability measure with density $(p,\vec u)\mapsto\frac1{4\pi} \langle\tilde n_p,\vec u\rangle^+$ on $S^1\times S^1$.
\footnote{In the limit the authorized vectors are the unit vectors entering the ball.}

\begin{center}
\includegraphics[scale=0.4]{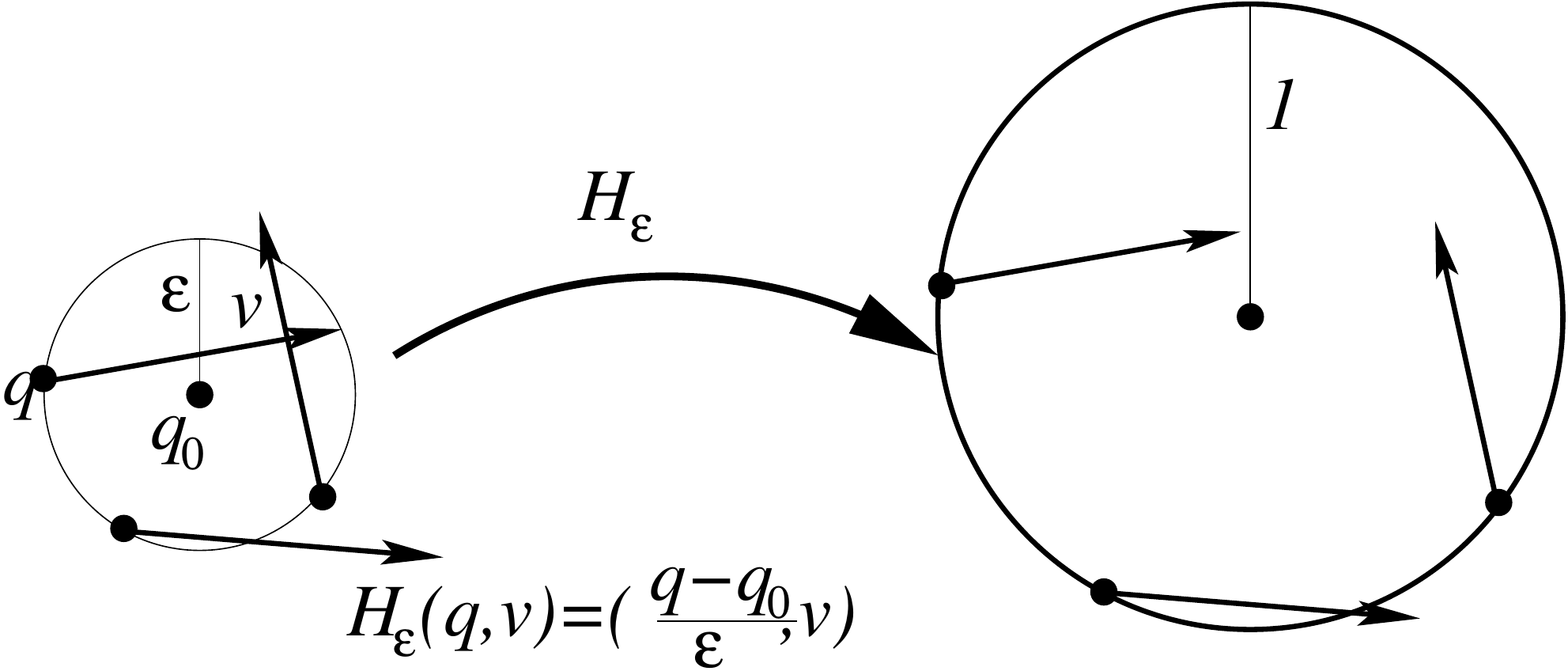}
\end{center}

\end{itemize}
\end{theorem}
\begin{proof}
We apply Theorem \ref{THMflow} to go from the discrete time to the continuous time.
Let $\mathcal A_{\varepsilon}:=B(x,\varepsilon)$ (resp. $\mathcal A_{\varepsilon}:=B(q,\varepsilon)\times S^1$). Return times to these sets have already been studied in \cite{ps10}.
We set $A_\varepsilon:=\Pi \mathcal A_\varepsilon$ and study the discrete time process associated to these sets. To this end, we apply Proposition~\ref{appli1} after checking its assumptions.
\begin{itemize}
\item 
We know that $\mu(A_\varepsilon)=2\varepsilon^2/|\partial Q|$ due to \cite[Lemma 5.1]{ps10} (resp. $\mu(A_\varepsilon)=2\pi\varepsilon/|\partial Q|$ due to \cite[Lemma 5.1]{ps10}). 
So $d=2$ (resp. $d=1$). Let $\sigma<d$ and $\delta>1$.
\item 
Note that $\mu((\partial A_\varepsilon)^{[\varepsilon^\delta]})
=o(\mu(A_\varepsilon))$.
Due to \cite[Theorem 3.3]{ps10} (resp.  \cite[Lemma 6.4]{ps10}, $  \mu(\tau_{A_\varepsilon}\le\eps^{-\sigma}|A_\varepsilon)=o(1)$.
\item We define $\mathcal B_\varepsilon:=\{(q,\vec v)\in\partial B(x_0,\varepsilon)\times S^1\ :\ \langle\tilde n_q,\vec v\rangle\ge 0\}$.
We endow it with the measure $\tilde\mu$ given by
$d\tilde\mu(q,\vec v)=\cos \varphi\, dr\, d\varphi$ with $\varphi=\angle(\tilde n_q,\vec v)$ and $r$ the curvilinear 
abscissa of $q$ on $\partial B(x_0,\varepsilon)$.
\item 
Let $\tau_{\mathcal A_\varepsilon}(y):=\inf\{t>0\ :\ Y_t(y)\in \mathcal A_\varepsilon\}$.
\item We define $H_\varepsilon:A_\varepsilon\rightarrow S^1\times[-1,1]$
which maps $x=(q,\vec v)\in A_\varepsilon$ to $\varepsilon^{-1}(\Pi_Q(Y_{\tau_{\mathcal A_\varepsilon}}(x))-q_0,\angle (\vec v_0,\vec v))$ (resp. $H_\varepsilon:A_\varepsilon\rightarrow S^1\times S^1$
which maps $x=(q,\vec v)\in A_\varepsilon$ to $(\varepsilon^{-1}\Pi_Q(Y_{\tau_{\mathcal A_\varepsilon}}(x))-q_0),\vec v)$).
\item Note that the image measure of $\mu(\cdot| A_\varepsilon)$
by $x\mapsto Y_{\tau_{\mathcal A_\varepsilon}}(x)$
corresponds to $\tilde\mu(\cdot | Y_{\tau_{\mathcal A_\varepsilon}}(A_\varepsilon))$. 
The set $Y_{\tau_{\mathcal A_\varepsilon}}(A_\varepsilon)$
consists of points of $\mathcal B_\varepsilon$
such that $\angle(\vec v_0,v)<\varepsilon$ (resp. $Y_{\tau_{\mathcal A_\varepsilon}}(A_\varepsilon)=\mathcal B_\varepsilon$).
\item  Hence $\mu(H_\varepsilon^{-1}(\cdot)|A_\varepsilon)$
is the image measure of $\mu(\cdot|Y_{\tau_{\mathcal A_\varepsilon}}(A_\varepsilon))$ by $(q,\vec v)\mapsto \varepsilon^{-1}(q-q_0,\langle(\vec v_0,\vec v))$
(resp. $\mu(H_\varepsilon^{-1}(\cdot)|A_\varepsilon)$
is the image measure of $\mu(\cdot|\mathcal B_\varepsilon)$
by $(q,\vec v)\mapsto (\varepsilon^{-1}q,\vec v)$).

Hence we obtain the convergence in distribution of
these families of measures.
\item For the construction of $\mathcal W$ we use Proposition~\ref{partition}.
\end{itemize}
\end{proof}
\section{Successive visits in a small neighborhood of an hyperbolic periodic point}\label{periodic}

\subsection{General results around a periodic hyperbolic orbit}
We consider the case of a periodic point $x_0$ of smallest period $p$ (i.e. $p$ is the smallest $n>0$ such that $T^nx_0=x_0$).

By periodicity, returns to $B(x_0,\varepsilon)$ appear in clusters, so that we cannot
hope that the return process is represented by a simple Poisson process. However, the occurrence of clusters
 should be well separated and have a chance to be represented by a simple Poisson Process.

Thus we define $A_\varepsilon$ as the set of points
of $B(x_0,\varepsilon)$ leaving $B(x_0,\varepsilon)$ for a time at least $q_0$, i.e.
$$
A_\varepsilon:=B(x_0,\varepsilon)\setminus \bigcup_{j=1}^{q_0} T^{-jp}
B(x_0,\varepsilon)
$$ 
and consider $\mathcal{N}_\eps$ defined by \eqref{PointProcess} with this choice of $A_\eps$. This definition of $A_\eps$ essentially records the last passage among a series of hitting to the ball.
We emphasize that in general, one has to consider $q_0>1$ to avoid clustering of occurrences of $A_\varepsilon$ due to finite time effects\footnote{
Indeed, consider the determinant one hyperbolic matrix $M =\begin{pmatrix}-0.2 & 1.8\\ 0.6 & -0.4\end{pmatrix}$. 
The vector
$v=\begin{pmatrix}0.5\\ 0.7\end{pmatrix}$ belongs to the unit ball $B_1$, $Mv\not\in B_1, M^2v\in B_1$ and $M^3v\not\in B_1$.
Assume that $T$ preserves the Lebesgue measure $\mu$ and has a fixed point $x_0$ such that $D_{x_0}T=M$.
Let $A_\varepsilon=B(x_0,\varepsilon)\setminus T^{-1}B(x_0,\varepsilon)^c$. 
One easily shows that the inequality 
\[
\mu(A_\varepsilon\cap T^{-2}A_\varepsilon) \ge \mu(B(x_0,\varepsilon)\cap T^{-1} B(x_0,\varepsilon)^c \cap T^{-2}B(x_0,\varepsilon)\cap T^{-3}B(x_0,\varepsilon)^c),
\]
contradicts the assumption that $\Delta(\{A_\eps\})=o(\mu(A_\eps))$.
}.

\begin{lemma}\label{LEM0}
Assume $x_0$ is a hyperbolic fixed point of $T$ and that $T$ is $C^{1+\alpha}$ in a neighborhood $U$ of the orbit $x_0,\ldots,T^{p-1}x_0$.
Then there exist an integer $q_0$ and $a>0$ such that for any $\varepsilon>0$ sufficiently small,
for any $n=1,\ldots,\lfloor a\log 1/\varepsilon\rfloor $, $A_\varepsilon \cap T^{-n}A_\varepsilon=\emptyset$, where $A_\varepsilon:=B(x_0,\varepsilon)\setminus \bigcup_{j=1}^{q_0} T^{-jp}
B(x_0,\varepsilon)$.
\end{lemma}

For $\varepsilon>0$ small enough, we define, as in Theorem \ref{ThmgeneAppli1},
$H_\varepsilon:B(x_0,\varepsilon)\mapsto B(0,1)$ by $H_\varepsilon:=\varepsilon^{-1}\exp^{-1}$.
We are interested in the behaviour of the point processes $(\tilde{\tilde{\mathcal N}}_\varepsilon)_\varepsilon$ defined by
\begin{equation}\label{tildetildeNeps}
 \tilde{\tilde{\mathcal N}}_\varepsilon(x):=\sum_{n\ :\ T^n(x)\in B(x_0,\varepsilon)} \delta_{(n\mu( B(x_0,\varepsilon)),H_\varepsilon(T^n(x)))}.
\end{equation}
The extremal index defined by 
\[
\theta_\eps = \frac{\mu(A_\eps)}{\mu(B(x_0,\varepsilon))}
\]
relates the above process to
\begin{equation}\label{tildeNeps}
 \tilde{\mathcal N}_\varepsilon(x):=\sum_{n\ :\ T^n(x)\in B(x_0,\varepsilon) \setminus \{x_0\}} 
 \delta_{(n\mu( A_\varepsilon),H_\varepsilon(T^n(x)))}
\end{equation}
 in an obvious way.
Indeed, $\tilde{\tilde{ {\mathcal N}}}_\eps=\Theta_\eps(\tilde{\mathcal N}_\eps)$
where 
\begin{equation}
\Theta_\eps \left(\sum_n \delta_{(t_n,x_n)}\right)=\sum_n \delta_{(\theta_\eps^{-1}t_n,x_n)}.
\end{equation}
We see these processes as point processes on $[0,+\infty)\times \dot B(0,1)$ (where $\dot B(0,1)=B(0,1)\setminus\{0\}$ is the open punctured ball). 

Assume that $\Omega$ is a $d$-Riemmannian manifold.
A $p$-periodic point of $T$ is said to be hyperbolic if $T^p$
defines a $C^1$ diffeomorphism between two neighborhoods of $x_0$ and if $DT_{x_0}$ admits no eigenvalue of modulus 1.
We write $E^u_{x_0}$ for the spectral space associated
to eigenvalues of modulus strictly larger than 1.
\begin{theorem}\label{compoundPoisson}
Assume that $T^{-1}$ is well defined on a small neighborhood of $x_0$ and
that $(\mathcal N_\varepsilon)_\varepsilon$ converges in distribution to a Poisson point process $\mathcal P$ with intensity $\lambda\times m$,
then the sequence of point processes $(\tilde{\mathcal N}_\varepsilon)_\varepsilon$
converges in distribution to the point process $\mathcal N=\Psi(\mathcal P)$ on $[0,+\infty)\times \dot B(0,1)$, with
$$\Psi\left(\sum_{n}\delta_{(t_n,x_n)}\right):= 
      \sum_{n\, :\, x_n\ne 0} \sum_{k=0}^{\ell_{x_n}}  \delta_{(t_n,DT_{x_0}
^{-kp}(x_n))}\, ,$$
where $\ell_y:=\inf\{k\ge 0\ :\ DT_{x_0}^{-kp}(y) \in B(0,1)\setminus
      \bigcup_{j=1}^{q_0}DT_{x_0}^{jp}B(0,1)\}$.
\end{theorem}
\begin{proof}
Note that $m$ has support in $ B(0,1)\setminus
      \bigcup_{j=1}^{q_0}DT_{x_0}^{-jp}B(0,1)$.

Observe that, for every $\varepsilon$ small enough, $\tilde{\mathcal N}_\varepsilon(x)$ is the image measure of 
$\mathcal N_\varepsilon(x)$ by
$$\Psi_\varepsilon:\sum_{n}\delta_{(t_n,x_n)}\mapsto   \sum_{n} 
  \sum_{k=0}^{\ell_{x_n,\varepsilon}} 
    \delta_{(t_n-kp\mu(A_\varepsilon),H_\varepsilon T^{-kp}H_\varepsilon^{-1}(x_n))},$$
with
$$ \ell_{y,\varepsilon}
:=\inf\{k\ge 0\ :\ H_\varepsilon T^{-kp}(H_\varepsilon^{-1}y)\in B(0,1)\setminus \bigcup_{j=1 }^{q_0}T^{jp}B(0,1)\}.$$
Observe that, for every $x\in B(0,1)$,
$$ \lim_{\varepsilon\rightarrow 0}
\ell_{x,\varepsilon}=\ell_x\, ,$$
and that, for every $k$, every $t>0$ and every $x\in B(0,1)$, 
$$\lim_{\varepsilon\rightarrow 0} (t-kp\mu(A_\varepsilon),H_\varepsilon T^{-kp}H_\varepsilon^{-1}(x))=
   (t,DT_{x_0}^{-kp}(x)).$$
Let us consider a set $R=\cup_i ]r_i,s_i[\times F_i$ such that $\mathbb E[\mathcal N(\partial R)]=0$ where $0\le r_i< s_i\le r_{i+1}$ and where $F_i$ are open precompact subsets of 
$\dot B(0,1)$.
Note that, since the closure of $\cup_iF_i$ does not contain 0, there exists $K>0$ and $\varepsilon_0>0$ (depending on the $F_i$'s) such that
$$\sup_{\varepsilon\in(0,\varepsilon_0),\, y\in\cup_i F_i}\inf\{k\ge 0\ :\ H_\varepsilon^{-1}T^{kp}H_\varepsilon(y)\in A_\varepsilon\}<K\, $$
and
$$\sup_{\varepsilon\in(0,\varepsilon_0),\, y\in\cup_i F_i}\inf\{k\ge 0\ :\ H_\varepsilon^{-1}T^{-kp}H_\varepsilon(x)\in B(0,1)\setminus
      \bigcup_{j=1}^{q_0}T^{jp}B(0,1)\}<K\, $$
so that 
$$\forall y\in \bigcup_i F_i, \left[\exists x\in A_\varepsilon,\ \exists k\in \{0,...,\ell_{x,\varepsilon}\},\  y=H_\varepsilon T^{-kp}(x)\right]\ \ \Rightarrow \ \  \ell_{x,\varepsilon}\le 2K\, .$$
By definition of $\ell_{x,\varepsilon}$,
for every $y\in  \bigcup_i F_i$, every $x\in A_\varepsilon$, every $k\in \{0,...,\ell_{x,\varepsilon}\}$ such that $y=H_\varepsilon T^{-kp}(x)$, we have:
$$x=T^{\tau^{(0)}_{A_{\varepsilon}}} (y)$$
with $\tau^{(0)}_A(y):=\inf\{k\ge 0\ :\ T^k(y)\in A\}$.
Due to lemma \ref{LEM0}, there exists $\varepsilon_1\in(0,\varepsilon_0)$ such that, for every $\varepsilon\in(0,\varepsilon_1)$, 
$$\forall y\in \bigcup_i F_i,\quad \{T^k(y),\ k=0,...,K\}\cap A_\varepsilon=\{T^{\tau^{(0)}_{A_{\varepsilon}}} (y)\} \, .$$
Therefore, for every $\varepsilon\in (0,\varepsilon_1)$
$$ \tilde{\mathcal N}_\varepsilon (R)=\tilde\Psi_{\varepsilon,K}(\mathcal N_\varepsilon) (R) =\mathcal N_{\varepsilon}\left(\bigcup_{k=0}^K\varphi_{\varepsilon,k}^{-1}(R)\right)\, ,$$
with
$$\tilde{\Psi}_{\varepsilon,K}:\sum_{n}\delta_{(t_n,x_n)}\mapsto   \sum_n 
  \sum_{k=0}^{K} 
    \delta_{(t_n-kp\mu(A_\varepsilon),H_\varepsilon T^{-kp}H_\varepsilon^{-1}(x_n))} $$
and with $\varphi_{\varepsilon,k}:(t,x)\mapsto (t-kp\mu(A_\varepsilon),H_\varepsilon T^{-kp}H_\varepsilon^{-1}(x))$.
Arguing analogously for $\mathcal N$ and $\mathcal P$ instead of
$\tilde{\mathcal N}_\varepsilon$ and $\mathcal N_\varepsilon$,
we obtain 
$$\mathcal N(R)=\tilde\Psi_{K}(\mathcal P) (R) =\mathcal P\left(\bigcup_{k=0}^K\varphi_{k}^{-1}(R)\right)\, $$
and
$$(\lambda\times m)(\bigcup_{k=0}^K\varphi_k^{-1}(\partial R)))=\mathbb E[\mathcal N(\partial R)]=0\, ,$$
with
$$\tilde{\Psi}_{K}:\sum_{n}\delta_{(t,x)}\mapsto   \sum_n 
  \sum_{k=0}^{K} 
    \delta_{(t_n,H_\varepsilon T^{-kp}H_\varepsilon^{-1}(x_n))}\, .$$
and
$$ \varphi_k:(t,x)\mapsto (t\mu(A_\varepsilon),DT_{x_0}^{-kp}(x))\, .$$
Since $(\mathcal N_\varepsilon)_\varepsilon$ converges in distribution to $\mathcal P$, we conclude that
$\left(\mathcal N_\varepsilon\left(\bigcup_{k=0}^K\varphi_{k}^{-1}(R)\right)\right)_\varepsilon $
converges in distribution to $\mathcal P\left(\bigcup_{k=0}^K\varphi_{k}(R)\right) $.
Moreover 
$$\left|\mathcal N_\varepsilon\left(\bigcup_{k=0}^K\varphi_{k}^{-1}(R)\right)-\mathcal N_\varepsilon\left(\bigcup_{k=0}^K\varphi_{\varepsilon,k}(R)\right)\right|\le  \mathcal N_\varepsilon\left(\bigcup_{k=0}^K\varphi_{k}^{-1}(\partial R^{[\eta_\varepsilon]})\right)$$
with $\lim_{\varepsilon\rightarrow 0}\eta_\varepsilon=0$.
Since $(\mathcal N_\varepsilon)_\varepsilon$ converges in distribution
to $\mathcal P$ and since $\mathcal N(\partial R)=0$, we conclude that
$$\left(\mathcal N_\varepsilon\left(\bigcup_{k=0}^K\varphi_{k}^{-1}(R)\right)-\mathcal N_\varepsilon\left(\bigcup_{k=0}^K\varphi_{\varepsilon,k}(R)\right)\right)_\varepsilon$$
converges in distribution to 0.
Therefore $(\tilde{\mathcal N}_\varepsilon(R))_\varepsilon$
converges in distribution to $\mathcal N(R)$.
\end{proof}

Theorem~\ref{compoundPoisson} is designed for invertible systems, near a hyperbolic periodic point, and does not apply to expanding maps. Indeed, in such a non-invertible situation, one has to define the set $A_\eps$ with the first passage in the ball $B(x_0,\eps)$ and not the last. More precisely, one has to set 
$$A_\eps=T^{-(q_0+1)}B(x_0,\eps)\setminus \cup_{j=1}^{q_0}T^{-j} B(x_0,\eps).
$$
We leave to the reader the generalization of 
Theorem~\ref{compoundPoisson} and the result of the next section to this case.

\subsection{SRB measure for Anosov maps}

We now consider a $C^2$ Anosov map $T$ on the $d$-dimensional riemaniann manifold $\Omega$. 
We assume that the measure $\mu$ is the SRB measure of the system \cite{kh},
and that $x_0$ is a periodic point of $T$ of smallest period $p$.

\begin{theorem}
We assume that $\mu(B(x_0,2\eps))\eps^{b_0} = o(\mu(B(x_0,\eps)))$ for some $b_0>0$ sufficiently small\footnote{Indeed this happens when for example the pointwise dimension of $\mu$ at $x_0$ exists and is bounded away from $0$ and $\infty$.}.
Then 
\begin{enumerate}
\item[(a)]
The point process $\N_\eps$ for entrances in $A_\eps$ is asymptotically Poisson $\mathcal P_\eps$, of intensity $\lambda\times m_\eps$
\item[(b)]
The point process $\tilde\N_\eps$ for entrances in $B(x_0,\eps)$ is asymptotically $\Psi(\mathcal P_\eps)$.
\item[(c)]
The point process $\tilde{\tilde{\N}}_\eps$ for entrances in $B(x_0,\eps)$ is asymptotically $\Theta_\eps(\Psi(\mathcal P_\eps))$.
\item[(d)]
The return time point process 
$$\mathcal {T}_\eps:=\sum_{n\colon T^nx\in B(x_0,\eps)}\delta_{n\mu(B(x_0,\eps))}$$ 
is asymptotically the compound Poisson point process $\pi\Theta_\eps(\Psi(\mathcal P_\eps))$,
where $\pi$ is the projection on the time axis.
\end{enumerate}
\end{theorem}

The compound Poisson distribution (d) has already been established for few dynamical systems~\cite{ht,hv,fft,cnz}, typically with strong assumptions on the dimension (1 or 1+1), the measure (non singular)  and the shape of the balls (e.g. products of stable and unstable balls).
We point out that our result is valid for balls $B(x_0,\eps)$ in the original Riemmanian metric and with a possibly singular measure; Note that the convergence of the extremal exponent $\theta_\eps$ is not expected with singular measures.

\begin{proof}
Due to the proof of Theorem \ref{THM1}, we assume that $(m_\eps)_\eps$ converges to some $m$.
(b) will follow from (a) by Theorem \ref{compoundPoisson}. 
The fact that (b) implies (c), which implies (d) comes from the definitions.

Let's prove (a) by applying Proposition~\ref{appli1}. It is well known that our system satisfies Hypothesis~\ref{HHH}-(II) for any $\alpha,\beta>0$. Our assumption on $(m_\eps)_\eps$
ensures Assumption (iii) of Proposition~\ref{appli1}. 

Choose $b_0,b$ such that $0<b_0<b<a d_u\log\lambda^{-1}$, with $a$ given by Lemma~\ref{LEM0}.
Let $p_\eps=\eps^{-\sigma}$ with $\sigma=b-b_0$. By Lemma~\ref{LEM0}
$$
\mu(A_\eps \cap \{\tau_{A_\eps}\le p_\eps\}) \le \sum_{n=\lfloor a\log1/\eps\rfloor+1}^{p_\eps}\mu(A_\eps \cap T^{-n}A_\eps).
$$
By Lemma~\ref{epsb} (with $c=1$) this sum is bounded by
\[
p_\eps \eps^{b}\mu B(x_0,2\eps)=o(\mu(B(x_0,\eps)))
\]
by assumption. Hence assumption \eqref{a} of Proposition~\ref{appli1} holds by Lemma~\ref{extremal}. (ii) comes from Lemma \ref{lem:couronne} and (iv) follows as in the proof of Theorem~\ref{ThmgeneAppli1}.
\end{proof}

\begin{lemma}\label{epsb}
For any $a,b,c>0$ such that  $a\log\lambda+b/d_u+c<1$, for $\varepsilon$ small enough, for any $n\ge a\log1/\varepsilon$ we have
\[
\mu( A_\varepsilon \cap T^{-n}A_\varepsilon ) \le \eps^b \mu(B(x_0,\varepsilon+\varepsilon^c)).
\]
\end{lemma}
\begin{proof}
Let $\kappa>0$ small and consider a partition (or a cover with finite multiplicity) of $\Omega$ 
by pieces of unstable cubes (or balls) $W\in \mathcal{W}$ of size $\varepsilon^{1-\kappa}$.
Let $\mathcal{V}$ be the set of the $V=T^{-n}W$, $W\in \mathcal{W}$. 

We can disintegrate the measure with respect to $\mathcal{V}$ such that for any set $Z$
\begin{equation}\label{disintegrate}
\mu(Z ) = \int_\mathcal{V} \mu(Z|V)d\mu(V).
\end{equation}
Since each $W=T^nV$ is a small piece of a smooth manifold, its intersection with $B(x_0,\varepsilon)$ consists at most in an unstable ball of radius $C \varepsilon$. Therefore the proportion of the unstable volume of $W\cap B(x_0,\varepsilon)$ in $W$ is bounded by $C \varepsilon^{\kappa d_u}$, where $d_u$ is the unstable dimension. By distortion, we get that 
$\mu(T^{-n}B(x_0,\varepsilon)|V) \le C \varepsilon^{\kappa d_u}$.

Using \eqref{disintegrate} we get
\[
\begin{aligned}\mu(B(x_0,\varepsilon)\cap T^{-n}B(x_0,\varepsilon))
&= 
\int_{\mathcal{V}} \mu(B(x_0,\varepsilon)\cap T^{-n}B(x_0,\varepsilon)|V) d\mu(V)\\
&\le 
\int_{\mathcal{V}} \mu(T^{-n}B(x_0,\varepsilon)|V) 1_{V\cap B(x_0,\varepsilon)\neq\emptyset}d\mu(V) \\
&\le  C \varepsilon^{\kappa d_u} \mu(B(x_0,\varepsilon+\varepsilon^c)),
\end{aligned}
\]
since $\diam V\le C\lambda^n \varepsilon^{1-\kappa} \le \varepsilon^c$.
\end{proof}

\begin{lemma}\label{7/8} There exists a constant $\delta_0>0$ independent of $\eps$ such that
\[
\mu(B(x_0,\frac78\eps)) \le (1-\delta_0) \mu(B(x_0,\eps)).
\]
\end{lemma}
\begin{proof}
We fix a measurable partition $\V$ of unstable manifolds such that the $d_u$-dimensional Lebesgue measure of $V$ is bounded from below by a constant, and $x_0\not\in\overline{\cup_V\partial V}$,
such that the disintegration \eqref{disintegrate} holds true.

Let $\rho=\frac78$.
Suppose that $V\in \mathcal{V}$ intersects the ball $B(x_0,\rho\eps)$.
Then $V$ intersects also the sphere $S(x_0,\frac{\rho+1}2\eps)$ in a point say $x$, and
the ball $B(x,\frac{1-\rho}2\eps)\cap V$ is contained in $B(x_0,\eps)\setminus B(x_0,\rho\eps)$.
Since $\mu(\cdot|V)$ is equivalent to the $d_u$-dimensional Lebesgue measure, there exists $\delta_0>0$ such that 
\[
\mu( B(x_0,\eps)\setminus B(x_0,\rho\eps)|V)
\ge 
\mu( B(x,\frac{1-\rho}2\eps)| V) \ge \delta_0 \mu(B(x_0,\eps)|V)
\]
and the lemma follows by integration.
\end{proof}

\begin{lemma}\label{extremal}
The extremal index is bounded away from zero: $\theta_\eps\ge\frac{\delta_0}{q_0+1}>0$.
\end{lemma}
\begin{proof}
We first observe, using equation~\eqref{Tq} that
\[
B(x_0,\eps)\cap \{\tau_{B(x_0,\eps)} \circ T^{q_0}\le q_0\} \subset T^{-q}B(x_0,\frac78\eps).
\]
for some integer $q$. Denote for simplicity $B=B(x_0,\eps)$.
By Lemma \ref{7/8} this gives
\[
\mu(B\cap \{\tau_{B}\circ T^{q_0}> q_0\}) \ge \delta_0 \mu(B).
\]
Recall that 
\[
\mu(A_\eps)= \theta_\eps \mu(B).
\]
We have
\[
B \cap  \{\tau_{B}\circ T^{q_0}> q_0\} \subset \bigcup_{j=0}^{q_0}T^{-j}A_\eps
\]
hence
\[
\mu(B \cap  \{\tau_{B}\circ T^{q_0}> q_0\})\le (q_0+1)\theta_\eps \mu(B).
\]
This implies that $\delta_0\le (q_0+1)\theta_\eps$.
\end{proof}

\begin{lemma}\label{lem:couronne}
Suppose that $T$ is $C^{1+\alpha}$ and that $\alpha>1-1/\overline{d}_{\mu}(x_0)$.
Then
\[
\mu(B(x_0,\eps)\setminus B(x_0,\eps-\eps^\delta) = o( \mu(B(x_0,\eps) ))
\]
 for any $\delta$ such that $\overline{d}_\mu(x_0)<\frac{\delta+1}{2}<\frac1{1-\alpha}$.
\end{lemma}

\begin{proof}
We fix a measurable partition $\V$ of unstable manifolds such that the $d_u$-dimensional Lebesgue measure of $V$ is bounded from below by a constant, and $x_0\not\in\overline{\cup_V\partial V}$,
such that the disintegration \eqref{disintegrate} holds true.

Up to applying the exponential map at $x_0$ we assume that $B(x_0,\eps_0)$ is the ball $B(0,\eps_0)$ of $\RR^d$.

Let $\eps\in(0,\eps_0)$.
Let $V\in \V$. Let $y\in V\cap B(0,\eps)\setminus B(0,\eps-\eps^\delta)$.
Up to a rotation we suppose that locally $V$ is the graph of a $C^{1+\alpha}$ 
map $\varphi$ from $U\subset\RR^{d_u}$ to $\RR^{d-d_u}$, with $d_x\varphi=0$ where $x\in U$ is such that $y=(x,\varphi(x))$.
We have 
\begin{equation}\label{Beps}
(\eps-\eps^\delta)^2 \le |x|^2 + |\varphi(x)|^2 \le \eps^2.
\end{equation}
Let $h\in \RR^{d_u}$ such that $x+h\in U$ and $(x+h,\varphi(x+h))\in B(0,\eps)\setminus B(0,\eps-\eps^\delta)$. We have
\begin{equation}\label{Beps+h}
(\eps-\eps^\delta)^2 \le |x+h|^2 + |\varphi(x+h)|^2 \le \eps^2.
\end{equation}
Subtracting \eqref{Beps} to \eqref{Beps+h} we get, setting $a(h) = |\varphi(x+h)|^2-|\varphi(x)|^2$, since $\eps^{2\delta}\le \eps^{1+\delta}$,
\[
3\eps^{1+\delta}
\ge 
\left||x+h|^2 -|x|^2 + a(h)\right|
=
\left| |h|^2 + 2 x\cdot h+ a(h)\right|.
\]
Assume that $|h|>6\eps^\frac{1+\delta}2$. 
We fix a unit vector $u\in \RR^{d_u}$ and seek for the solutions $h=tu$ of the above equation, which becomes
\[
 \left|  t^2 + 2 x\cdot u t +a(tu) \right| \le 3 \eps^{1+\delta}.
\]
Therefore, dividing by $|t|>6\eps^\frac{1+\delta}2$ we end up with
\[
 \left|  t + 2 x\cdot u  +\frac{a(tu)}t \right| \le \frac12 \eps^\frac{1+\delta}2.
\]
Let $t,t'$ be two solutions. We obtain by subtraction
\[
 \left|  t'-t  +\frac{a(t'u)}{t'} - \frac{a(tu)}t \right| \le \eps^\frac{1+\delta}2.
\]
Note that the function $g_u(t):=\frac{a(tu)}t$ is $C^1$ in the range of $t$'s, and
\[
g_u'(t) = \frac{1}{t}(2 (d_{x+tu}\varphi  u)\cdot\varphi(x+tu)) - \frac{a(tu)}{t^2}.
\]
Since $\varphi$ is $C^{1+\alpha}$ we have $|d_{x+tu}\varphi u|= O( |t|^\alpha)$. In addition, 
\[
a(tu)
=(|\varphi(x+tu)|-|\varphi(x)|)(|\varphi(x+tu)|+|\varphi(x)|)
=O(\eps |t|^{1+\alpha}).
\]
This implies that $|g_u'(t)| = O(\eps |t|^{\alpha-1}) = O(\eps^{1+(\alpha-1)\frac{1+\delta}2})=o(1)$, hence
\[
 \left|  (t-t') (1+o(1)) \right| \le \eps^\frac{1+\delta}2.
\]
Using radial integration this gives that the $d_u$-dimensional Lebesgue measure of those $h$ such that $(x+h,\varphi(x+h))\in B(0,\eps)\setminus B(0,\eps-\eps^\delta)$ is $O(\eps^\frac{1+\delta}2)$.
Hence $\mu(B(x_0,\eps)\setminus B(x_0,\eps-\eps^\delta)|V)= O(\eps^\frac{1+\delta}2)$, and the result follows by integration using \eqref{disintegrate}.
\end{proof}

\section{Billiard in a diamond}\label{sec:bd}
We consider a diamond shaped billiard, with no cusp.
The billiard table $Q$ is a bounded closed part of $\mathbb R^2$ delimited by $4$ convex obstacles $(\Gamma_i)_{i\in\mathbb Z/4\mathbb Z}$ (with $C^3$-smooth boundary, with positive curvature) placed in such a way that, for every $i\in\mathbb Z/4\mathbb Z$, $\partial\Gamma_i$ meets $\partial \Gamma_{i+1}$ transversely at some point called corner $C_i$, but has no common point with $\partial\Gamma_{j}$ for $j\ne i-1,i+1$. In our representation of this billiard table, $C_1=(0,0)$ is on the left side of $Q$ and
the inner bisector at the corner $C_1$ is horizontal.
We consider again the billiard flow $(\mathcal M,\nu,(Y_t)_t)$
and the billiard map $(\Omega,\mu,T)$ in the domain $Q$.

For any $\varepsilon>0$, we put a virtual vertical barrier $I_\varepsilon$ of length $\varepsilon$ joining a point $a^{(1)}_\varepsilon\in\partial\Gamma_1$ to a point $a^{(2)}_\varepsilon\in\partial
   \Gamma_2$.
and we are interested in the times at which the billiard flow enters the
corner by crossing the barrier $I_\varepsilon$. So that we define
$$\mathcal A_\varepsilon:=I_\varepsilon\times S^1_-,\quad\mbox{with}\quad 
S^1_-:=\{\vec v=(v_1,v_2)\in S^1\ :\ v_1<0\} \, .$$
We take $V:=\mathbb R\times S^1_-$ and
$$\mathcal H_\varepsilon:(q=(q_1,q_2),\vec v)\mapsto \left(\frac{q_2}{\varepsilon},\vec v\right) .$$

\includegraphics[scale=0.5]{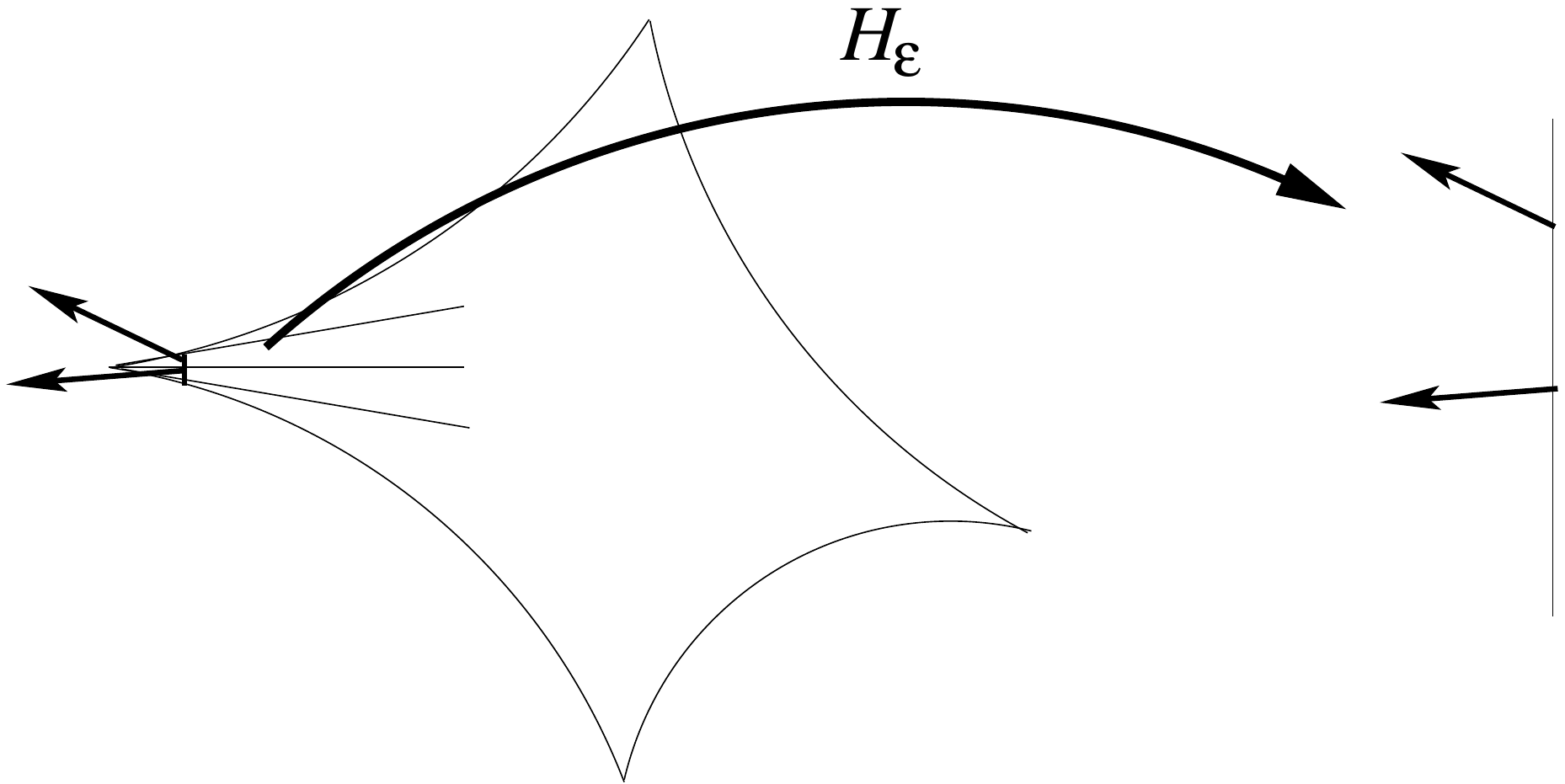}

\begin{theorem}\label{billardDiamant}
The point process
$$\sum_{t>0\, :\, Y_t\in \mathcal A_\varepsilon}\delta_{\left(\frac{\eps t}{\pi\, Area(Q)},\mathcal H_\varepsilon(Y_t(\cdot))\right)} $$
converges in distribution to a Poisson point process on $[-1/2,1/2]\times S^1_-$ with density $\lambda\times m_0$, with $m_0$
the probability measure of density proportional to $(\bar q,\vec v)\mapsto|v_1|$ with $\vec v=(v_1,v_2)$.
\end{theorem}

\subsection{Notations, recalls and proof of Theorem \ref{billardDiamant}}

Let us recall some useful facts and notations.

Due to the transversality of $\partial\Gamma_1$ and $\partial\Gamma_2$
at $C_1$, there exist $0<\theta_1<\theta_2$ such that, for $\varepsilon>0$ small enough,
the distance on $\partial Q$ between $C_1$ and $a^{(i)}_\varepsilon$ 
is between $\theta_1\varepsilon$ and $\theta_2\varepsilon$.

Here $\Omega$ is the set of reflected unit vectors based on $\partial Q\setminus\{
  C_1,...,C_4\}$. We parametrize $\Omega$ by $\bigcup_{i\in\mathbb Z/4\mathbb Z}
     \{i\}\times ]0,\length(\partial\Gamma_i\cap Q)[\times\left[-\frac \pi 2;\frac \pi 2\right]$.
A reflected vector $(q,\vec v)$ is represented by $(i,r,\varphi)$ if $q\in\partial\Gamma_i$
at distance $r$ (on $Q\cap \partial\Gamma_i$) of $C_{i-1}$ and if $\varphi$
is the angular measure in $[-\pi/2,\pi/2]$ of $(\vec{n}(q),\vec{v})$ where
$\vec{n}(q)$ is the normal vector to $\partial Q$ at $q$.

For any $C^1$-curve $\gamma$ in $\Omega$, we write $\ell(\gamma)$
for the euclidean length in the $(r,\varphi)$ coordinates of $\gamma$.
If moreover $\gamma$ is given in coordinates by $\varphi=\phi(r)$, then we also write $p(\gamma):=\int_\gamma \cos(\phi(r))\, dr$.
We define the time until the next reflection in the future by
$$\tau^+(q,\vec {v}):=\min\{s>0\ :\ q+ s\vec{v}\in \partial Q\}\, .$$
We also define $\tau^-:\Omega\rightarrow (0,+\infty)$ for the time until the last reflection in the past (corresponding to $\tau^-=\tau^+\circ T^{-1}$ when $T^{-1}$ is well defined) by
$$\tau^-(q,\vec {v}):=\min\{s>0\ :\ q+ s\vec{v}^-\in Q\}\, ,$$ with $\vec v^{-}$ be the reflected vector with respect to the normal to $\partial Q$ at $q$, i.e. $\vec v^-$ is the unit vector satisfying the following angular equality $\angle (\vec n(q),\vec v)=\angle(\vec v^-,\vec n(q))$.
It will be useful to define $R_0:=\{\varphi=\pm\pi/2\}$, $\mathcal C_+=\{(q,\vec v)\in \Omega\ :\ q+\tau^+(q,\vec v)\vec v \in\{C_1,...,C_4\}\}$ and $\mathcal C_-=\{(q,\vec v)\in \Omega\ :\ q+\tau^-(q,\vec v)\vec v^-\in\{C_1,...,C_4\}\}$. Observe that, for every $k\ge 1$, $T^k$ defines a $C^1$-diffeomorphism from $\Omega\setminus \mathcal S_{-k}$ to $\Omega\setminus \mathcal S_k$ with $\mathcal S_{-k}:=
   T^{-k}R_0\cup\bigcup_{m=0}^{k-1}T^{-m}(R_0\cup \mathcal C_+)$ and
$\mathcal S_{k}:=
   T^{k}R_0\cup\bigcup_{m=0}^{k-1}T^{m}(R_0\cup \mathcal C_-)$.
As for the other billiard models, we set $\pi_Q:\Omega\rightarrow Q$ for the canonical projection.

Despite the absence of the so called complexity bound in billards with corners, De Simoi and Toth have shown in \cite{dst} that some expansion condition holds, from which the 
growth lemma~\cite[Theorem 5.52]{ChernovMarkarian} follows. It says that for any weakly homogeneous unstable curve $W$ one has
\begin{equation}\label{GL}
m_W(r_n<\delta) \le c\theta^n \delta + c\delta m_W(W)
\end{equation}
where $m_W$ is the one dimensional Lebesgue measure on $W$, and $r_n(x)$ denotes the distance (on $T^nW$) of $T^n(x)$ to the boundary of the homogeneous piece of $T^nW$ containing $x$.

In particular, for those systems one can build a Young tower with exponential parameters, from which it follows that Hypothesis \ref{HHH}-(II) is satisfied
for every $\alpha,\beta>0$.

\begin{proof}[Proof of Theorem \ref{billardDiamant}]
Let $A_\varepsilon\subset \Omega$ be the set of
all possible configurations at the reflection time just after the particle crosses the virtual barrier $I_\eps$ from the right side.
Note that $\mu(A_\eps)=\frac 1{2|\partial Q|}\int_{[-\eps/2,\eps/2]\times[-\pi/2,\pi/2]}\cos\varphi\,  dr\, d\varphi=\frac{\eps}{|\partial Q|}$.

Due to Theorem \ref{THMflow}, it is enough to prove that
$(\mathcal N_\eps)_\eps$ converges to a Poisson point process
with density $\lambda\times m_0$ for the above choice of $A_\eps$ and for $ H_\eps:A_\eps\rightarrow \mathbb R\times S^1_-$, such that 
given by $\mathcal H_\eps=H_\eps\circ T\circ \Pi$ with $\Pi$ the projection defined in Subsection \ref{sec:specialflow}. To this end we apply Proposition \ref{appli1}. 

The fact that
Assumption (i) of Proposition \ref{appli1} is satisfied for $p_\varepsilon=\varepsilon^{-\sigma}$ comes from Proposition \ref{propshort}.  We take $\alpha>3/\sigma$.
Assumption (ii) comes from the fact that the boundary of each connected component of $\partial A_\varepsilon$
is made of a part of $R_0$ and of a $C^1$-increasing curve
$r=R(\varphi)$ with $R'(\varphi)=1/(\kappa(r)+\frac{\cos(\varphi)}{\tau^-(r,\varphi)})\le 1/\min\kappa $ corresponding to reflected vectors in the corner on the leftside of $I_\eps$
coming from $\{a_{\varepsilon}^{(1)},a_\varepsilon^{(2)}\}$
and to $T(R_0\cap A_\eps)$.
The image measure of $\mu(\cdot|A_\eps)$ by $H_\eps$
is proportional to $\cos(\varphi) drd\varphi$ where $r$ is the position
on $[-1/2,1/2]$ and $\varphi$ the angle (in $[-\pi/2,\pi/2]$) between the vector $(-1,0)$ and the incident vector (i.e. the speed vector at the time when the particle crosses $I_\eps$).
We take for $\mathcal W$ the set 
%of finite union 
of rectangles of the form $(a,b)\times(c,d)$ in the above $(r,\varphi)$ coordinates. 
Outside the strips $A_\eps\cap\{(r,\varphi)\:\ |r-a_\eps^{(i)}|<\eps^2\}$, $H_\eps$ is $K.\eps^{-2}$-Lipschitz (indeed the jacobian is in $O(\eps^{-1}/\cos\varphi)\le c\, \eps^{-2}$) and so, using the argument of the proof of Proposition \ref{partition}-(ii), we conclude that (iv) is satisfied since $\mu(A_\eps)=O(\eps)$.
\end{proof}

\subsection{Short returns}
The aim of this subsection is the following result.
\begin{proposition}\label{propshort}
There exists $\sigma>0$ such that
 $\mu(\tau_{A_\eps}<\eps^{-\sigma}|A_\eps)=o(1)$.
\end{proposition}
To this end, we will recall useful facts and introduce some notations.
Let $\tau_0:=\min_{(i,j)\:\ j\ne i,i+1}\dist(C_i,\Gamma_j)/10$.
\begin{definition}
We say that a curve $\gamma$ of $\Omega$ satisfies
assumption (C) if it is 
given by $\varphi=\phi(r)$ with $\phi$
$C^1$-smooth, increasing and such that $\min\kappa\le \phi'\le \max\kappa+\frac 1{\tau_0}$.
\end{definition}
We recall the following facts.
\begin{itemize}
\item There exist $C_0,C_1>0$ and $\lambda_1>1$ such that, for every 
$\gamma$ satisfying Assumption (C) and every integer $m$ such that
$\gamma\cap\mathcal S_{-m}=\emptyset$, $T^m\gamma$ is a $C^1$-smooth
curve satisfying
$C_1 p(T^m\gamma)\ge \lambda_1^mp(\gamma)$ and $\ell(\gamma)\le C_0
      \sqrt{p(T\gamma)}$.
\item There exist $C_2>0$ and $\lambda_2>\lambda_1^{1/2}$ such that, for every integer $m$,
the number of connected components of $\Omega\setminus \mathcal S_{-m}$ is less than
$C_2\lambda_2^m$. Moreover $\mathcal S_{-m}$ is made of curves $\varphi=\phi(r)$
with $\phi$ $C^1$-smooth and strictly decreasing.
\item If $\gamma\subset \Omega\setminus \mathcal S_{-1}$ is given by $\varphi=\phi(r)$ or $r=\mathfrak r(\varphi)$
with $\phi$ or $\mathfrak r$ increasing and $C^1$ smooth, then $T\gamma$
is $C^1$, is given by $\varphi=\phi_1(r)$ with
$\min\kappa\le \phi_1'\le\max\kappa+\frac 1{\min_\gamma\tau^+}$.
Moreover $\int_{T\gamma}\, d\varphi\ge \int_{\gamma}\, d\varphi$.
\item There exists $m_0$ such that, for every $x\in \Omega\setminus\bigcup_{k=0}^{m_0-1}
     T^{-k}\mathcal C_+$, there exists $k\in\{1,...,m_0\}$ such that $\tau^+(T^{k-1})>\tau_0$.
\end{itemize}

Let $A_\varepsilon\subset \Omega$ be the set of possible configurations of a particle 
at the reflection time just after the particle reaches the virtual barrier $I_\varepsilon$
from the right side. We observe that there exists $K_0>0$ and $\varepsilon_0>0$ such that, for every $\varepsilon\in(0,\varepsilon_0)$, for every $q$ between $C_1$
and $a^{(i)}_\varepsilon$, the set of $\vec v$ such that $(q,\vec v)\in A_\varepsilon$
has Lebesgue measure at least $K_0$.

\begin{lemma}[Very quick returns]\label{lemmaveryquick} 
There exists $K_1>0$ such that, 
$$\forall s\ge 1,\quad \mu(T^{-s}(A_\varepsilon)|A_\varepsilon)\le K_1(\lambda_2/\lambda_1^{\frac 12})^s\varepsilon^{\frac 12}.$$
\end{lemma}
\begin{proof}
Let $q\in\pi_Q(A_\varepsilon)$.
Let $\gamma_1$ be a connected component of  $\pi_Q^{-1}(\{q\})\setminus\mathcal S_s$. We define $\gamma:=\gamma_1\cap A_\varepsilon\cap T^{-s}A_\varepsilon$.
Let $m$ be the smallest positive integer such that $\min_\gamma\tau\circ T^{m-1}>\tau_0$. By definition of $A_\varepsilon$ and of $m_0$, $m<\min(m_0,s)$.
Hence $T^m\gamma$ satisfies Assumption (C) and
$$\ell(\gamma)\le \int_{\gamma}\, d\varphi  \le  \int_{T^{m}\gamma}\, d\varphi
 \le \ell(T^m\gamma)\le C_0\sqrt{p(T^m\gamma)}.$$
Moreover, since $\gamma\cap S_{-s}=\emptyset$, we also have
$$ p(T^m\gamma)\le C_1\lambda_1^{m-s}p(T^s\gamma).$$
But, since $T^s\gamma$ is an increasing curve contained in $A_\varepsilon$,
we conclude that $p(T^s\gamma)\le \theta\varepsilon$.
Hence
$$\ell(\gamma)\le C_0\sqrt{C_1\lambda_1^{m-s}\theta\varepsilon} .$$
By using the fact that
$\pi_Q^{-1}(\{q\})\setminus\mathcal S_s$ contains at most $C_2\lambda_2^s$
connected components and by integrating on $\pi_Q(A_\varepsilon)$, we obtain
$$ \mu(T^{-s}(A_\varepsilon)\cap A_\varepsilon)\le \frac {C_0\sqrt{\theta C_1}C_2
      \lambda_1^{\frac m2}(\lambda_2/\lambda_1^{\frac 12})^s\varepsilon^{\frac 12}}{2\length(\partial Q)}\length(\pi_Q(A_\varepsilon)).$$
We conclude by using the fact that $\mu(A_\varepsilon)\ge (1-\sin K_0) \length(\pi_Q(A_\varepsilon))$ and by setting $K_1:=\frac{C_0\sqrt{\theta C_1}C_2
      \lambda_1^{\frac {m_0}2}}{1-\sin K_0}$.
\end{proof}

\begin{lemma}[quick returns]\label{lemmaquick} 
For any $a>0$, there exists $s_a>0$ such that
\[
\sum_{n=-a\log\eps}^{\eps^{-s_a}} \mu(A_\eps\cap T^{-n}A_\eps)=o(\mu(A_\eps)).
\]
\end{lemma}
\begin{proof}
We take a measurable partition $\mathcal{V}$ of $\Omega\setminus \mathcal{S}_1$ by the unstable curves $\frac{d\varphi}{dr}=\kappa(r)$.
By disintegration there exist
a probability measure $\tilde\mu$ on $\mathcal{V}$, and a constant $c<\infty$ such that for any measurable set $B$ we have
\[
c^{-1}\mu( B) \le \int_{\mathcal{V}}m_W( B)d\tilde\mu(W)=:\mu_0(B).
\]
We define $\tilde A_\eps$ as the set of $T^{-j}x$ where $x\in A_\eps$ and $j$ is the minimal integer such that $T^{-\ell}x$ does not belong to the sides adjacent to the corner. Any corner sequence is bounded by some constant $m$ depending only on the billiard table, thus
$A_\eps \subset \cup_{\ell=1}^m T^{\ell}\tilde A_\eps$. 
It follows by invariance that 
\[
\mu(A_\eps \cap T^{-n} A_\eps) \le m^2 \max_{|\ell|\le m} \mu(\tilde A_\eps \cap T^{-n-\ell}\tilde A_\eps).
\]
Therefore it suffices to control $\mu(\tilde A_\eps \cap T^{-n}\tilde A_\eps)$.

Note that $\tilde A_\eps$ is at a long distance from the corner, hence there are finitely many  decreasing curves $\varphi_j$, $j=1..m_0$ and a constant $c$ such that $\tilde A_\eps \subset \tilde V_\eps$ where 
\[
\tilde V_\epsilon = \cup_j \{(r,\varphi)\colon  \varphi_j(r)-c\eps< \varphi <\varphi_j(r)+c\eps\}\, 
\]
where $(r,\varphi_j(r))$ represents a vector pointing exactly to the corner provided its two adjacent obstacles are removed. In particular $\frac{d\varphi_j(r)}{dr}\le-\kappa$.
We denote the $k$th homogeneity strip\footnote{See \cite{ChernovMarkarian} for notations and definitions.}  by $\mathbb{H}_k$ for $k\neq0$ and 
set $\mathbb{H}_0 = \cup_{|k|< k_0} \mathbb{H}_k$ for some fixed $k_0$.
Set $s:=\min(-a\log\theta,1)/3$. 
Let $k_\eps=\eps^{-s}$ and $H^\eps=\cup_{|k|\le k_\eps}\mathbb{H}_k$.
For any $W\in\mathcal{V}$ we set $W^\eps=W\cap \tilde V_\eps$, $W_k=W\cap \mathbb{H}_k$ and $W_k^\eps=W_k\cap \tilde V_\eps$. Each $W_k^\eps$ is a weakly homogeneous unstable curve.

We cut each curve $W_k^\eps$ into small pieces $W_{k,i}^\eps$ such that 
each $T^j W_{k,i}^\eps$, $j=0,\ldots,n$ is contained in a homogeneity strip and a connected component of $\Omega\setminus \mathcal{S}_1$. 
For $x\in W_{k,i}^\eps$ we denote by $r_n(x)$ the distance (in $T^n W$) of $T^n(x)$ to the boundary of $T^n W_{k,i}^\eps$.

By definition of $W^\eps$ 
\[
\begin{split}
m_W&(\tilde A_\eps \cap T^{-n} \tilde A_\eps)
\le 
m_{W^\eps}(T^{-n}\tilde V_\eps)\\
\le 
&m_{W^\eps}(\mathbb{H}_\eps^c) + \sum_{|k|\le k_\eps} m_{W_k^\eps}(\{r_n\ge \eps^{1-s}\}\cap T^{-n}\tilde V_\eps) + m_{W_k^\eps}(r_n<\eps^{1-s}).
\end{split}
\]
The first term inside the sum is bounded by the sum $\sum_i m_{W_{k,i}^\eps}(T^{-n}\tilde V_\eps) $ over those $i$'s such that $T^nW_{k,i}^\eps$ is of size larger than $\eps^{1-s}$. In particular $m_{T^n W_{k,i}^\eps}(T^n W_{k,i}^\eps)\ge \eps^{1-s}$. On the other hand, by transversality 
\[
m_{T^n W_{k,i}^\eps}(\tilde V_\eps) \le c\eps.
\]
By distortion (See Lemma 5.27 in \cite{ChernovMarkarian}) we obtain
\[
m_{W_{k,i}^\eps}(T^{-n}\tilde V_\eps) \le c \eps^{s} m_{W_{k,i}^\eps}(W_{k,i}^\eps).
\]
Summing up over these $i$ gives the first term inside the sum is bounded by
\[
m_{W_k^\eps}(\{r_n\ge \eps^{1-s}\}\cap T^{-n}\tilde V_\eps) 
\le c \eps^{s} m_{W_{k}^\eps}(W_{k}^\eps).
\]
On the other hand, the growth lemma~\eqref{GL} implies that 
\[
m_{W_k^\eps}(r_n<\eps^{1-s}) \le c \theta^n \eps^{1-s} + c\eps^{1-s} m_{W_k^\eps}(W_k^\eps).
\]
A final summation over $k$ gives
\[
m_W(\tilde A_\eps \cap T^{-n} \tilde A_\eps)
\le m_{W}(\tilde V_\eps \cap \mathbb{H}_\eps^c) + c(\eps^{s}+\eps^{1-s})m_W(\tilde V_\eps) + ck_\eps\theta^n \eps^{1-s}.
\]
Integrating over $W\in\mathcal{V}$ gives
\[
\mu(\tilde A_\eps \cap T^{-n} \tilde A_\eps) 
\le \mu_0(\tilde V_\eps \cap \mathbb{H}_\eps^c)
+ 
O(\eps^{1+s/3})
= O(\mu(A_\eps)\eps^{s/3}),
\]
where we used the fact that $\mu_0$ is equivalent to Lebesgue and $\tilde V_\eps \cap \mathbb{H}_\eps^c$  is contained in the union of at most $m_0$ rectangles of width $O(\eps)$ and height $k_\eps^{-2}=\eps^{2s}$. We take $s_a=s/6$. 
\end{proof}
\begin{proof}[Proof of Proposition \ref{propshort}]
Choose $a=1/(4\log(\lambda_2/\lambda_1^{1/2})$.
Observe that, due to Lemma \ref{lemmaveryquick}, we have
$$\sum_{s=1}^{-a\log\eps}\mu(T^{-s}A_\eps|A_\eps)\le
\frac{K_1}{\lambda_2/\lambda_1^{\frac 12}-1}(\lambda_2/\lambda_1^{\frac 12})^{-a\log\eps}\varepsilon^{1/2} \le 
\frac{K_1}{\lambda_2/\lambda_1^{\frac 12}-1}\varepsilon^{1/4} .$$
This combined with Lemma \ref{lemmaquick} leads to
$$\sum_{n=1}^{\varepsilon^{-s_a}}\mu(T^{-n}A_\eps|A_\eps)=o(1)\, . $$
\end{proof}
\appendix

\section{More proofs}

\begin{proof}[Proof of Proposition~\ref{appli1}]
Item (ii) of Theorem 
\ref{THM} being satisfied by assumption by $\mathcal W$, it 
 remains to prove Item (i).
 
Let $\mathcal G_0=\{G_1,...,G_L\}$ be a finite subcollection of $H_\varepsilon^{-1}\mathcal W$.
Let $t>0$.
Let $A\in \sigma(\mathcal G_0)$ and
$B\in\sigma(\bigcup_{n= 1}^{N_{\varepsilon,t}}
T^{-n} \mathcal G_0)$, with $N_{\varepsilon,t}:=\lfloor t/\mu(A_\varepsilon)\rfloor$.

Set $X_j:=(1_{T^{-j}(G_1)},\ldots,1_{T^{-j}(G_L)}) \in \RR^L$.
Note that $\mathbf 1_B=g(X_1,...,X_{N_{\varepsilon,t}})$ for some $g\colon(\{0,1\}^L)^{N_{\varepsilon,t}}\to\{0,1\}$.
Note that if $\varepsilon$ is small enough, $N_{\varepsilon,t}>p_{\varepsilon}$ by assumption.

Let $B_1=\{g(0,...,0,X_1,...,X_{N_{\varepsilon,t}-p_\varepsilon})=1\}$ so that
$ |\mathbf 1_B-\mathbf 1_{B_1}\circ f^{p_\varepsilon}|\le\mathbf 
   1_{\{\tau_{A_\varepsilon}\le p_\varepsilon\}}$.
Note that
$$ |\mu(B\cap A|A_\varepsilon)-\mu(B_1\cap A|A_\varepsilon)|\le  \mu(\tau_{A_\varepsilon}\le p_\varepsilon|A_\varepsilon)=o(1).$$
Moreover
$$|\mu(B)-\mu(B_1)|\le \mu(\tau_{A_\varepsilon}\le p_\eps)\le p_\varepsilon\mu(A_\varepsilon)=o(1).$$

Set $K:=\lfloor p_\varepsilon/4\rfloor$.

Under (I), set $M=0$, $\mathcal P_K=\mathcal Q_K$.

Under (II), set $M=K$, $\mathcal P_K=\mathcal Q_{2K}$.

It remains to show that 
\begin{equation}
|\mu(B_1\cap A)-\mu(B_1)\mu(A)|=o(\mu(A_\varepsilon)),
\end{equation}
i.e.
$$|\mathrm{Cov}_{\tilde\mu}(\mathbf 1_{\tilde T^{-M},\tilde\Pi^{-1}B_1},\mathbf 1_{ \tilde T^{-M}\tilde\Pi^{-1}A})|=o(\tilde\mu(\tilde T^{-M}\tilde\Pi^{-1}A_\varepsilon)).
$$

 We approximate $\tilde T^{-M}\tilde\Pi^{-1}A$ (resp. $\tilde T^{-M}\tilde\Pi^{-1}A_\varepsilon$) by the union
of the atoms of $\mathcal P_K$ intersecting this set. We write $\tilde A$ (resp. $\tilde A_\varepsilon$) for this union. Note that
$\tilde T^{-M}\tilde\Pi^{-1}A\subset \tilde A$ (resp. $\tilde T^{-M}\tilde\Pi^{-1}A_\varepsilon\subset \tilde A_\varepsilon$) and that
$$
\tilde A\setminus \tilde T^{-M}\tilde\Pi^{-1}A\subset
A':=\bigcup_{Q\in Q_K:\tilde\Pi \tilde T^MQ\cap A\ne\emptyset,\tilde\Pi \tilde T^MQ\setminus B_1\ne\emptyset}Q$$
and so (due to the assumption on the diameters of the atoms of $\mathcal P_K$),
\begin{eqnarray*}
\tilde\mu(A')
&\le&\tilde\mu\left(\bigcup_{Q\in Q_K:\tilde\Pi\tilde T^M Q\subset (\partial A)^{[Ck^{-\alpha}]}}Q\right)\\
&\le&\mu((\partial A)^{[CK^{-\alpha}]})=o(\mu(A_\varepsilon)).
\end{eqnarray*}
Analogously
$$
\tilde A_\varepsilon\setminus \tilde T^{-M}\tilde\Pi^{-1}A_\varepsilon\subset 
A'_\varepsilon:=\bigcup_{Q\in Q_K:\tilde\Pi \tilde T^M Q\cap A_\varepsilon\ne\emptyset,\tilde\Pi T^MQ\setminus A_\varepsilon\ne\emptyset}Q$$
and
$$
\tilde\mu(A'_\varepsilon)
=o(\mu(A_\varepsilon)).
$$
Finally we approximate $ \tilde T^{-M}\tilde\Pi^{-1}B_1 $ by a union $\tilde B$
of atoms of $\sigma(\bigcup_{\ell=1}^{N_{\varepsilon,t}-p_\varepsilon}\mathcal P_K)$ such that $ \tilde T^{-M}\tilde\Pi^{-1}B_1 \subset\tilde B_1$ and
$$\tilde B_1\setminus  T^{-M}\tilde\Pi^{-1}B_1
\subset B_1':=\bigcup_{\ell=1}^{N
   _{\varepsilon,t}-p_\varepsilon} F^{-\ell}\left(\bigcup_{Q\in \mathcal P_K:\tilde \Pi\tilde T^M Q\subset \bigcup_{k=1}^L(\partial G_k)^{[CK^{-\alpha}]}}Q\right),$$
and
$$\tilde\mu( B_1')=(N_{\varepsilon,t}-p_\varepsilon) o(\mu(A_\varepsilon)).$$
Therefore
$$|Cov_{\tilde\mu}(\mathbf 1_{\tilde\Pi^{-1}A},\mathbf 1_{ \tilde\Pi^{-1}B_1})-Cov_{\tilde\mu}(\mathbf 1_{\tilde A},\mathbf 1_{ \tilde B_1})|
$$
\begin{eqnarray*}
&\le&|Cov_{\tilde\mu}(\mathbf 1_{\tilde A}-\mathbf 1_{T^{-M}\tilde\Pi A},\mathbf 1_{\tilde B_1} )+Cov_{\tilde\mu}(\mathbf 1_{T^{-M}\tilde\Pi A},\mathbf 1_{\tilde B_1}-\mathbf 1_{T^{-M}\tilde\Pi B_1} )|\\
&\le&|Cov_{\tilde\mu}(\mathbf 1_{A'},\mathbf 1_{\tilde B} )|+|Cov_{\tilde\mu}(\mathbf 1_{\tilde A},\mathbf 1_{B_1'})|\\
&\ &+2\tilde\mu(A')\tilde\mu(\tilde B_1)+2
\tilde\mu(\tilde A)\tilde\mu(B_1')\\
&\le& C'_0(\mu(A_\varepsilon)p^{-\beta}+o(\mu(A_\varepsilon))
(N-p)\mu(A_\varepsilon))=o(\mu(A_\varepsilon)).
\end{eqnarray*}
\end{proof}

\begin{proof}[Proof of Lemma~\ref{LEM0}]
Since the analysis is local, taking the exponential map at $x_0$ if necessary we may assume that $U\subset \RR^d$. Let $L_0$ be the Lipschitz norm of $T$.
We first observe that if $a>0$ is sufficiently small, then for any $\varepsilon>0$ small
and $n\le a\log1/\varepsilon$, $\|x-x_0\|<\varepsilon$ and $\|T^nx-x_0\|<\varepsilon$
 implies that 
\[
\|T^nx_0-x_0\| \le \|T^nx-T^nx_0\|+\|T^nx-x_0\| \le L_0^n\eps+\eps,
\] 
and thus $n$ is a multiple of $p$. Hence without loss of generality we assume that $p=1$.

Let $q$ be the integer given by Lemma~\ref{matrice} for the hyperbolic matrix $D_{x_0}T$.

Since the $C^{1+\alpha}$ norm of $T^n$ is growing at most exponentially fast, changing the value of $a>0$ if necessary, there exists $c,L\ge L_0$ such that for any $x\in B(x_0,\varepsilon)$,
\begin{equation}\label{dl1+alpha}
\|T^nx-T^nx_0 -D_{x_0}T^n(x-x_0) \| \le cL^n \varepsilon^{1+\alpha}\le \varepsilon/2,
\end{equation}
for any $1\le n\le a\log1/\varepsilon$.

Let $2q\le n\le a\log(1/\varepsilon)$ and suppose that $x\in B(x_0,\varepsilon)\cap T^{-n}B(x_0,\varepsilon)$.
Using \eqref{dl1+alpha} we obtain $\|D_{x_0}T^n(x-x_0)\| \le \|T^nx-x_0\| + \varepsilon/2 <\frac32 \varepsilon$. Thus Lemma~\ref{matrice} gives $ \|D_{x_0}T^q(x-x_0)\|\le\frac38 \varepsilon$.
Using again~\eqref{dl1+alpha} we get 
\[
\|T^qx-T^q x_0\| \le \|D_{x_0}T^q(x-x_0)\|+\varepsilon/2 < \frac78\varepsilon.
\]
Hence
\begin{equation}\label{Tq}
B(x_0,\eps)\cap T^{-n}B(x_0,\eps)\subset T^{-q}B(x_0,\frac78\eps).
\end{equation}
Set $q_0=2q$. When $n\ge q_0$ this proves that 
$$
A_\varepsilon\cap T^{-n}A_\varepsilon \subset A_\eps \cap T^{-q}B(x_0,\eps)=\emptyset,
$$ 
while when $n<q_0$ the intersection is empty by definition.
\end{proof}

\begin{lemma}\label{matrice} Let $A$ be hyperbolic matrix. There exists an integer $q$ such that for any vector $v$, any $n\ge 2q$, $\|A^qv\|\le \max(\|v\|, \|A^nv\|)/4$.
\end{lemma}
We leave the proof to the reader.

%%%%%%%%%%%%%%%

\end{document}